\documentclass[reqno]{amsart}

\usepackage{amsmath,amssymb,amscd,accents,enumitem}
\usepackage{graphicx}
\usepackage[utf8]{inputenc} 

\usepackage[T2A,T1]{fontenc}

\usepackage{etoolbox}
\pretocmd{\section}{\numberwithin{Cor}{section}}{}{}
\pretocmd{\subsection}{\numberwithin{Cor}{subsection}}{}{}
\pretocmd{\subsubsection}{\numberwithin{Cor}{subsubsection}}{}{}

\DeclareSymbolFont{cyrillic}{T2A}{cmr}{m}{n}
\SetSymbolFont{cyrillic}{bold}{T2A}{cmr}{bx}{n}
\DeclareMathSymbol{\B}{\mathord}{cyrillic}{193}
\def\so{\raisebox{.5ex}{\scalebox{0.6}{\#}}\kern-.02em{}o}

\DeclareGraphicsExtensions{.eps} \usepackage{mathrsfs}
\usepackage[mathcal]{eucal} 
\usepackage{esint}
\usepackage[numbers,sort]{natbib}

\usepackage[pdfencoding=unicode, psdextra, pdfusetitle]{hyperref}

\usepackage{cleveref}

\usepackage{bookmark}
\renewcommand{\paragraph}[1]{\medskip\noindent#1}
\renewcommand{\subparagraph}[1]{\medskip\noindent\textit{#1}}
\newcommand{\mycite}[2]{\cite[#1]{#2}}

\newtheorem{Thm}{Theorem}{\bfseries}{\itshape}
{\bfseries}{\itshape}
\newtheorem{Cor}{Corollary}{\bfseries}{\itshape}
\newtheorem{ThmCor}[Cor]{Theorem}{\bfseries}{\itshape}
\newtheorem{Prop}[Cor]{Proposition}{\bfseries}{\itshape}
{\bfseries}{\itshape}
\newtheorem{Lem}[Cor]{Lemma}{\bfseries}{\itshape}
{\bfseries}{\itshape}
\newtheorem{Conj}[Thm]{Conjecture}{\bfseries}{\itshape}
\newtheorem{Def}[Cor]{Definition}{\bfseries}{\rmfamily}
{\bfseries}{\rmfamily}
{\scshape}{\rmfamily}
\newtheorem{Rem}[Cor]{Remark}{\scshape}{\rmfamily}
\newtheorem*{Rem*}{Remark}{\scshape}{\rmfamily}
{\scshape}{\rmfamily}
{\bfseries}{\itshape}

\newcommand{\iref}[2][]{%
  \def\tempA##1 ##2\relax{%
    \ifstrequal{##1}{Condition}%
        {Condition~(\ref{#1#2})}{%
    \ifstrequal{##1}{Assumption}%
        {Assumption~(\ref{#1#2})}{%
    \ifstrequal{##1}{Property}%
        {Property~(\ref{#1#2})}{%
    \ifstrequal{##1}{Case}%
        {Case~(\ref{#1#2})}{%
    \ifstrequal{##1}{Equation}%
        {\eqref{#1#2}}{%
    \ifstrequal{##1}{Section}
        {\S\ref{#1#2}}{%
    \ifstrequal{##1}{Subsection}
        {\S\ref{#1#2}}{%
    \ifstrequal{##1}{Subsubsection}
        {\S\ref{#1#2}}%
        {##1~\ref{#1#2}}%
  }}}}}}}}%
  \expandafter\tempA\string#2\relax
}

\def\size{{\operatorname{size}}}

\begin{document}

\title{Counting theorems for algebraic relations}

\author[G.~Binyamini]{Gal Binyamini}
\address{Weizmann Institute of Science, Rehovot, Israel}
\email{gal.binyamini@weizmann.ac.il}
\thanks{Funded by the European Union (ERC, SharpOS, 101087910), and by the ISRAEL SCIENCE FOUNDATION (grant No. 2067/23).}

\author[N.~Hirata-Kohno]{Noriko Hirata-Kohno}
\address{Nihon University, College of Science and Technology, Tokyo, Japan}
\email{hiratakohno.noriko@nihon-u.ac.jp}

\author[M.~Kawashima]{Makoto Kawashima}
\address{Meiji Gakuin University, Yokohama, Japan}
\email{kawashima\_makoto@mi.meijigakuin.ac.jp}

\author[Y.~Salant]{Yuval Salant}
\address{Weizmann Institute of Science, Rehovot, Israel}



\date{\today}

\begin{abstract}
Let $X$ be a set definable in a sharply o-minimal structure. We consider the problem of counting the number of points where $X$ intersects algebraic varieties $V$ over $\mathbb Q$ of dimension $k<\operatorname{codim} X$, as a function of $T:=\deg V+\operatorname{h}(V)$, where $\operatorname{h}(V)$ is the \emph{log}-height of $V$. In particular, we conjecture that after removing a suitable "algebraic part", this number grows polynomially in $T$ -- a generalization of Wilkie's conjecture. We show that this full conjecture implies some open problems in algebraic independence theory. We also formulate a weaker conjecture stating that all intersections above are contained in a $\operatorname{poly}(T)$ amount of balls of radius $e^{-T}$.

  We then consider the case where $X\subseteq\mathbb C^n$ is a (compact piece of a) trajectory of a polynomial differential equation satisfying a variant of Nesterenko's $D$-property. Our main theorem is a proof of the weakened conjecture for such curves when $k<\sqrt n-1$.
\end{abstract}
\maketitle

\section{Introduction}

\subsection{Synopsis}

In \cite{pw:thm} Pila and Wilkie introduced the \emph{Pila-Wilkie
  counting theorem}. Roughly, this states that if a set definable in an
o-minimal structure contains no positive-dimensional semialgebraic
subsets, then it can only contain a small number of rational points ---
\emph{sub-polynomial} in the Height. This was later extended in many
directions, including generalizing rational points to algebraic points
of fixed degree \cite{pila:blocks}.  These counting theorems have had
numerous applications in arithmetic geometry and functional
transcendence, see~\cite{pila:book} for a modern account. The original
Pila-Wilkie paper also included a conjecture known as \emph{Wilkie's
  conjecture}, stating that the Pila-Wilkie asymptotic can be improved
in some natural structures (originally $\mathbb R_{\exp}$) to
\emph{polylogarithmic} in the Height. This has been proven in the
original context and in more general structures in
\cite{me:pfaff-wilkie}.

The intuition that underpins the various counting theorems
(Pila-Wilkie, Wilkie's conjecture and various generalizations) is that
\emph{geometry governs arithmetic}: we do not expect to see many
arithmetic (e.g. rational) solutions to a system of equations unless
there is a geometric reason. The same philosophy suggests a
conjectural strengthening of the counting theorems for sets of
codimension greater than one. Namely, in the absence of natural
geometric obstructions, one expects that a set $X\subseteq\mathbb R^n$ should
meet very few algebraic subvarieties of dimension $k<\operatorname{codim} X$, as a
function of the degree and log-height of the varieties. Such a conjecture
is to algebraic independence theory as the Wilkie conjecture is to
transcendence theory, with the classical case being recovered when
$k=0$.

In this paper we precisely formulate the conjecture above and some
weaker open variants. We show that the full conjecture implies some
classical open problems in algebraic independence theory of
multiplicative groups. Our main result, Theorem~\ref{thm:main},
is proof of the weaker variant of the conjecture where
$X$ is equal to $\Gamma\subseteq\mathbb C^n$, the image of a trajectory of a polynomial differential
equation with a tameness condition (see
Definition~\ref{def:arith-d-prop}) and $k<\sqrt n-1$. The proof of
this result rests upon a result from a previous paper,
Theorem~\ref{Theorem order-bound}, which provides an essentially optimal
lower bound for the maximum of the norm of a polynomial $P$ restricted to
$\Gamma$ in terms of $\deg(P)+\operatorname{h}(P)$, where $\operatorname{h}(P)$ is the \emph{log}-height of $P$.

\subsection{Counting theorems: known and conjectural}

Let $X\subseteq\mathbb R^n$ be definable in some sharp o-minimal (\so-minimal) structure. We refer the reader to \cite{me:pfaff-wilkie,me:sharp-os-cells} for an introduction to this notion. The notion of \so-minimality was introduced in \cite{me:pfaff-wilkie} to provide an axiomatic framework for o-minimal structures where one can expect to prove polylogarithmic counting theorems. Before this formalism existed, Wilkie conjectured that a polylogarithmic counting theorem should hold in the structure ${\mathbb R}_{\exp}$, and the unfamiliar reader may consider this as an illustrative example. Roughly speaking, \so-minimality assigns to each definable subset two parameters called \emph{format} and \emph{degree} and asserts certain bounds on the geometric complexity of a set -- for instance on the number of connected components -- in terms of its format and degree. Since our goal is to formulate conjectures concerning more general polylogarithmic counting theorems, \so-minimality is the natural context in which to formulate our conjectures. Wilkie's conjecture (now theorem \cite{me:pfaff-wilkie}) is formulated below as a particular case of our more general conjecture.

For
$k=0,\ldots,n-1$ we denote by $\mathcal C^n_k$ the collection of all
varieties of $\mathbb C^n$ of dimension $k$ and by
$\mathcal C^n_{k,g,h}\subseteq\mathcal C^n_k$ the subset consisting of varieties over $\overline{\mathbb Q}$ of
degree at most $g$ and log-height at most $h$. Denote
\begin{equation*}
  X(k,g,h) := \{ x\in X : x\in W \text{ for some } W\in \mathcal C^n_{k,g,h}\}.
\end{equation*}
We note that $X(0,g,h)$ is simply the collection of all algebraic
points of degree $\leq g$ and log-height $\leq h$ in $X$, usually denoted
$X(g,h)$. Our goal is to count, under appropriate restrictions, the
number of points in $X(k,g,h)$.

To state our conjectures we also need a variant of the
\emph{algebraic} and \emph{transcendental} parts used in the
Pila-Wilkie theorem. For $k=0,\ldots,n-\dim X-1$, let
\begin{multline}
  \quad X^{\operatorname{alg}(k)} := \{ x\in X : x\in W \text{ for some }\\
  W\in\mathcal C^n_{k+1} \text{ with } \dim_x(X\cap W) \ge 1 \}.\quad
\end{multline}
The dimension above is the real dimension of $X\cap W$ at $x$, in the
sense of \so-minimality. In other words, this is the locus where $X$
intersects a variety of dimension $k+1$ in a positive-dimensional
locus of unexpectedly high dimension. In particular
$X^{\operatorname{alg}(0)}=X^\mathrm{alg}$ is the classical algebraic part in the sense of
Pila-Wilkie. We similarly denote
\begin{equation*}
  X^{\operatorname{trans}(k)}:=X- X^{\operatorname{alg}(k)}.
\end{equation*}
With these notations in place we can now state our main conjecture. Even though
we state the conjecture in the language of \so-minimal structures, the reader unfamiliar with this setup may consider the \so-minimal structure $\mathbb R_{\exp}$ as a representative example.

\begin{Conj}\label{conj:strong}
  Let $X\subseteq\mathbb R^n$ be definable in a \so-minimal structure
  with sharp derivatives and let $k=0,\ldots,n-\dim X-1$. Then, for every
  $g,h\in\mathbb N$,
  \begin{equation}\label{eq:conj-strong}
    \# X^{\operatorname{trans}(k)}(k,g,h) = \operatorname{poly}_X(g,h).
  \end{equation}
\end{Conj}

\begin{Rem}\label{rem:conj-strong-exp}
    In analogy with the known cases of Wilkie's conjecture, it is natural to conjecture an even stronger condition: namely that if $X$ has format $\mathcal{F}$ and degree $D$ in the \so-minimal structure then the polynomial appearing in the right hand side of~\eqref{eq:conj-strong} can be replaced by $\operatorname{poly}_{\mathcal{F}}(D,g,h)$. To keep in line with Wilkie's original formulation we have chosen to state a less precise form in Conjecture~\ref{conj:strong}.
\end{Rem}

Conjecture~\ref{conj:strong} with $k=0$ reduces to the
polylogarithmic counting theorem
\begin{equation*}
  \#X^{\operatorname{trans}}(g,h)= \operatorname{poly}_X(g,h).
\end{equation*}
This case was proved in \cite{me:pfaff-wilkie}, and is itself a mild generalization of
Wilkie's original conjecture. Another motivation for
Conjecture~\ref{conj:strong} is the following theorem, where we assert
polynomial dependence on $h$ but not on $g$.

\begin{Thm}\label{thm:count-fixed-g}
  Suppose $X\subseteq\mathbb R^n$ is definable in a \so-minimal structure $\mathcal{S}$ with sharp
  derivatives. Then for fixed $g\in\mathbb N$ and $k=0,\ldots,n-\dim X-1$ it
  must satisfy
  \begin{equation}\label{Equation count-fixed-g}
    \#X^{\operatorname{trans}(k)}(k,g,h) \le \operatorname{poly}_{X,g}(h).
  \end{equation}
\end{Thm}

This theorem is proved in \iref{Section count-fixed-g}. However, the
full Conjecture~\ref{conj:strong} appears to be far more
challenging. In particular we show in \iref{Section conj-trans-applications} that it
implies some open problems in algebraic independence theory for
multiplicative groups. Since the full conjecture seems to be out of
reach, we formulate a weaker form of the conjecture as a benchmark for
progress in this direction.

\begin{Conj}\label{conj:weak}
  Let $X\subseteq\mathbb R^n$ be definable in a \so-minimal structure
  with sharp derivatives and $k=0,\ldots,n-\dim X-1$. Then, for every
  $g,h\in\mathbb N$,
  \begin{equation*}
    X^{\operatorname{trans}(k)}(k,g,h) \subseteq \bigcup_\eta B_\eta
  \end{equation*}
  where $\{B_\eta\}_\eta$ is a collection of a $\operatorname{poly}_X(g,h)$ amount of balls of radius
  $e^{-(g+h)}$.
\end{Conj}

Existing techniques in algebraic independence theory seem to be more
suitable for establishing results in the spirit of
Conjecture~\ref{conj:weak} than the full
Conjecture~\ref{conj:strong}. Nevertheless, even
Conjecture~\ref{conj:weak} seems to be quite challenging.

\subsection{Preliminaries}
\label{Subsection prelim}

Before stating our main results we introduce some notations. Let $K$
be a number field and let $\xi$ denote a polynomial vector field on
$\mathbb C^n$ defined over $K$,
\begin{equation*}
  \xi = \sum \xi_i \frac{\partial}{\partial x_i}, \qquad \xi_i\in K[x_1,\ldots,x_n].
\end{equation*}
Let $D_r:=\{z\in\mathbb C:|z|<r\}$ be the open disk in $\mathbb C$ and let $\gamma:\overline{D_1}\to\mathbb C^n$ be a compact trajectory of $\xi$ which is analytic on $D_1$ and continuous on $\overline{D_1}$ (for our purposes, being a trajectory means that $\forall z\exists\alpha_z\in\mathbb C$ such that $\xi(\gamma(z))=\alpha_z\cdot\gamma'(z)$). We define $\Gamma:=\gamma(\overline{D_\frac1e})$ to be the radius $\frac1e$ image of this trajectory.

For $p\in\mathbb Z[x_1,\ldots,x_n]$ and $V\subseteq\mathbb C^n$ a pure-dimensional variety defined over $K$ (equivalently, we may look at the ideal $I_V$ that generates $V$), we denote
\begin{align}
  \operatorname{t}(p) &:= \deg p+\operatorname{h}(p), & \operatorname{t}(V)&:=\deg V+\operatorname{h}(V)
\end{align}
where $\deg(\cdot),\operatorname{h}(\cdot)$ denote the degree and the log of the Height (log-height) respectively (the Height of a polynomial over $\mathbb Z$ is the maximum modulus of its coefficients, and the Height of a variety is the Height of its Chow form). We may sometimes use these notations when $p^\mathrm{proj}\in\mathbb Z[X_0,...,X_n]$ is homogeneous or $V^\mathrm{proj}\subseteq\mathbb{P}^n(\mathbb C)$ as well (preserving the following identifications between projective and affine: $p(x_1,...,x_n):=p^\mathrm{proj}(1,x_1,...,x_n)$ and $V^\mathrm{proj}:=$ the projective closure of $V$). In the following definition we use the notation $|V(\omega)|$ with $\omega\in\mathbb C^n$ for the \emph{absolute value of evaluating} $V$ \emph{at} $\omega$, see \cite[Chapter 3, Definition 4.6]{NesterenkoPhilippon2001IntroductionToAlgebraicIndependenceTheory} for the definition (for the projectivization of the corresponding ideal that generates $V$, evaluated at $(1:\omega)$). This is an elimination theoretic notion that generalizes the idea of substituting $\omega$ in the polynomial generating a hypersurface to varieties of larger codimension (independently of the normalization of this polynomial).

\begin{Def}[Order along $\Gamma$]\label{def:order}
  We define the \emph{order} of a polynomial or variety along $\Gamma$ as
  \begin{align}
    \operatorname*{ord}_\Gamma p &:= -\log \max_{\omega\in\Gamma}\{|p(\omega)|\}, &  \operatorname*{ord}_\Gamma V &:= -\log \max_{\omega\in\Gamma}\{|V(\omega)|\}.
  \end{align}
\end{Def}

Our result will apply to vector fields satisfying the following
property inspired by Nesterenko's $D$-property.

\begin{Def}[Arithmetic $D$-property]\label{def:arith-d-prop}
  We say that $\xi,\Gamma$ satisfy the \emph{arithmetic $D$-property over $K$} if
  there exists a constant $C$ such that for every irreducible
  $K$-variety $V\subseteq\mathbb C^n$ which is contained in a non-trivial
  $\xi$-invariant variety,
  \begin{equation*}
    \operatorname*{ord}_\Gamma V \le C\cdot \operatorname{t}(V)^{\frac{n}{\operatorname{codim} V}}.
  \end{equation*}
\end{Def}

The arithmetic $D$-property automatically holds, for instance, when the
invariant subvarieties of $\xi$ are contained in finitely many
hypersurfaces, as was shown in \cite[\S 6.1.2]{BinyaminiSalant2026Technical}. This is the setup that appears
in Nesterenko's more classical $D$-property. More generally we have the
following criterion.

\begin{Prop}[\mycite{Proposition 6.1.2.1}{BinyaminiSalant2026Technical}]\label{Proposition d-prop-criterion}
  Suppose there exists a constant $C$ such that for every irreducible $K$-variety $V\subseteq\mathbb C^n$ which is contained in a non-trivial $\xi$-invariant variety,
  \begin{equation*}
    -\log\max_{\omega\in\Gamma}\operatorname{dist}(\omega,V) \leq C \cdot \operatorname{t}(V)^{\frac n {\operatorname{codim} V}-1}.
  \end{equation*}

where $\operatorname{dist}(\omega,V):=\min_{x\in V}\|x-\omega\|_{\infty}$.

Then $\xi,\Gamma$ satisfy the arithmetic $D$-property over $K$.
  
\end{Prop}

\iref{Proposition d-prop-criterion} can be used to verify the
arithmetic $D$-property in various cases where a good description of the
$\xi$-invariant varieties is available thanks to Ax-Schanuel type
functional transcendence statements.

\subsection{Main result on counting relations}

Our main theorem can now be stated as follows.

\begin{Thm}[see \iref{Theorem main theorem} and \iref{Remark main theorem} for more details]\label{thm:main}
  Let $\xi,\gamma$ be as in \iref{Subsection prelim} satisfying the
  arithmetic $D$-property over $K$. Then for any $\rho\geq 1,k<\sqrt n-1$,
  \begin{equation*}
    \Gamma(k,g,h) \subseteq \bigcup_\eta B_\eta
  \end{equation*}
  where $\{B_\eta\}_\eta$ is a collection of a $(g+h)^{\rho\cdot C_{\xi,\gamma,K}}$ amount of balls of radius $e^{-(g+h)^\rho}$ for some constant $C_{\xi,\gamma,K}$ depending only on $\xi,\gamma,K$.
  
\end{Thm}

We note that since $\Gamma^{\operatorname{trans}(k)}\subseteq\Gamma$ then Theorem~\ref{thm:main}
establishes Conjecture~\ref{conj:weak} for the trajectory $\Gamma$
with $k<\sqrt n-1$.

The proof of Theorem~\ref{thm:main} is based on an application of
Philippon's criterion for algebraic independence
\cite[Chapter 8, Main criterion for algebraic independence (CIA)]{NesterenkoPhilippon2001IntroductionToAlgebraicIndependenceTheory}. As is often the case, a key to being able to
apply this criterion is proving very sharp lower bounds for
polynomials restricted to $\Gamma$. In the following section we
describe an essentially optimal result of this type which plays the
main role in the proof of Theorem~\ref{thm:main}.

\subsection{Order estimates for polynomials}

The result from our previous paper is an essentially optimal estimate for the
order of a polynomial along $\Gamma$. Throughout this section we let
$\xi,\Gamma$ be as in \iref{Subsection prelim} and assume that they satisfy
the arithmetic $D$-property over $K$. We also let $p$ denote a polynomial in
$K[x_1,\ldots,x_n]$, with a coefficient that is $\geq 1$ in absolute value.

\begin{ThmCor}[The Lower Bound Theorem, see \mycite{Corollary 6.3.2}{BinyaminiSalant2026Technical}]\label{Theorem order-bound}
  There is a constant $C$ depending only on $\xi,\gamma,K$ such that
  \begin{equation*}
    \operatorname*{ord}_\Gamma p \le C\cdot \operatorname{t}(p)^n.
  \end{equation*}
\end{ThmCor}

This is a specific case of a more general statement for orders of varieties along
$\Gamma$:

\begin{ThmCor}[The Lower Bound Theorem, \mycite{Theorem 7}{BinyaminiSalant2026Technical}]\label{thm:variety-order-bound}
  There is a constant $C$ depending only on $\xi,\gamma,K$ such that for
  any irreducible $K$-variety $V$,
  \begin{equation*}
    \operatorname*{ord}_\Gamma V \le C\cdot \operatorname{t}(p)^{\frac{n}{\operatorname{codim} V}}.
  \end{equation*}
\end{ThmCor}

The following corollary of Theorem~\ref{Theorem order-bound}, proved using
standard complex variable methods, is particularly useful in the
context of algebraic independence theory.

\begin{Cor}[The Lower Bound Theorem, \mycite{Corollary 6.3.1}{BinyaminiSalant2026Technical}]\label{cor:pointwise-bound}
  There is a constant $C$ depending only on $\xi,\gamma,K$ such that for any $0<h<1$,
  \begin{equation*}
    \log |p(\omega)| \ge - C\cdot\log\frac1h \cdot \operatorname{t}(p)^n
  \end{equation*}
  for any $\omega$ in $\Gamma$ outside of the image of a union of $C\cdot \operatorname{t}(p)^n$ discs in $\mathbb C$ with the sum of the radii not exceeding $h$.
  
\end{Cor}

And the following corollary generalizes the classical multiplicity
estimate of Nesterenko
\cite{nesterenko:mult-nonlinear}, \cite{Nesterenko1996ModularFunctionsAndTranscendenceQuestions}, which gives the
upper bound $(\deg p)^n$ for the \emph{order of zero}
$\operatorname*{ord}_{\omega_0} p|_\Gamma$ at a given point $\omega_0\in\Gamma$:

\begin{Cor}[\mycite{Corollary 9}{BinyaminiSalant2026Technical}] \label{cor:zero-bound}
  There is a constant $C$ depending only on $\xi,\gamma, K$ such that

  \begin{equation*}
    \# \text{zeroes of }p\text{ in }\Gamma\text{ counted with multiplicities}  \le C\cdot \operatorname{t}(p)^n.
  \end{equation*}

\end{Cor}

This can be seen as a global version of Nesterenko's result; Corollary~\ref{cor:zero-bound}
counts the total number of zeroes in $\Gamma$ hence generalizing the
local estimate. We note however that Corollary~\ref{cor:zero-bound}
does not seem to be sufficient for deriving algebraic independence
results such as Theorem~\ref{thm:main}, which requires the stronger
estimate from Corollary~\ref{cor:pointwise-bound}.

A lot of work in the literature has been dedicated to the problem of
upper bounding the number of zeroes of a polynomial on the trajectory
of a vector field. Novikov and Yakovenko \cite{ny:chains} give a bound
for arbitrary polynomial vector fields but with exponential dependence
on $\deg p$. Yomdin and Comte \cite{yc:zeros-rational} have given
bounds polynomial in $\deg p$ when the trajectory passes through a
rational point. Binyamini \cite{me:poly-zeros} gives a bound depending
polynomially on $\deg p$ under a condition similar to the $D$-property,
with the exponent roughly $2n^2$. For the special class of
\emph{Pfaffian curves}, sharp bounds are available as a consequence of
Khovanskii's Fewnomial theory \cite{khovanskii:fewnomials}.

To our knowledge, Corollary~\ref{cor:zero-bound} is the first to
establish a bound with the optimal exponent $\operatorname{t}(p)^n$, albeit at the
cost of introducing dependence on the log-height of $p$. In algebraic
independence applications, this additional dependence does not
generally cause issues, while the optimal exponent $n$ plays a very
important role.

Our proof for Theorem~\ref{Theorem order-bound} is inspired by
Nesterenko's proof in the local case, and consists of generalizing
Nesterenko's machinery to our more global notion of order. This
introduces some difficulties associated to passing from Nesterenko's
non-archimedean order valuation to an archimedean valuation. We
address this by systematic application of the notion of the \emph{Bernstein index} (see \cite[\S 2.1]{BinyaminiSalant2026Technical}).

\subsection{Sketch of the proof of Theorem~\ref{thm:main}}

Our basic approach relies on the algebraic independence criterion of
Philippon \cite[Chapter 8,  Main criterion for algebraic independence (CIA)]{NesterenkoPhilippon2001IntroductionToAlgebraicIndependenceTheory}. The precise form we use is stated
in \iref{Subsection modified criterion}. Roughly speaking, in order to show that a point
$\omega\in\Gamma$ does not lie on \emph{any} algebraic variety $V$ of dimension $k$ over $\mathbb Q$, this criterion requires that we construct a sequence of polynomials $\{Q_N\}_{N\in\mathbb N}$ with $\operatorname{t}(Q_N)\to\infty$ with some regularity condition, such that
\begin{equation}\label{Equation sketch-cia}
  \operatorname{t}(Q_N)^{k+1} < -\log |Q_N(\omega)| < \operatorname{t}(Q_N)^{k+1+\varepsilon}
\end{equation}
for some sufficiently small $\varepsilon$ and $N\gg1$. We adapt Philippon's proof to give a quantitative form of this criterion: to avoid meeting varieties $V$ with $\operatorname{t}(V)<T$ it suffices to construct the sequence with $T^{c_1}<N<T^{c_2}$ for some suitable constants $c_1,c_2$.

Let $k$ by any integer that is smaller than $\sqrt n-1$. Suppose that for some $T\gg1$,
\begin{equation*}
  \Gamma(k,T^\alpha,T^\alpha) \ge T^\ell
\end{equation*}
where $\ell$ is some large constant to be determined and $\alpha$ is some small constant to be determined. If this is not
the case then the claim is proved. We construct polynomials vanishing
identically on these $T^\ell$ points: this is done by a pigeonhole
argument similar to the proof of Siegel's lemma, using an upper bound
for the Hilbert function of prime ideals in $\mathbb Q[x_1,\ldots,x_n]$ due
to Nesterenko \cite{Nesterenko1985EstimatesForTheCharacteristicFunctionOfAPrimeIdeal}. This enables us to construct polynomials $p_N$ with $\operatorname{t}(p_N)\simeq N$ and roughly $N^{n-k-\alpha}$ zeroes. It then follows from standard complex estimates that
\begin{equation}\label{Equation sketch-aux-upper}
  \log \max_{\omega\in\Gamma}\{|p_N(\omega)|\}\lesssim -N^{n-k-\alpha}.
\end{equation}
For the lower bound, we apply Corollary~\ref{cor:pointwise-bound} to
see that outside a $\operatorname{poly}(N)$ amount of discs of log-radius $N^\theta$ (for some sufficiently small $\theta$) we have
\begin{equation}\label{Equation sketch-aux-lower}
  \log |p_N(\omega)| \gtrsim -N^{n+\theta}.
\end{equation}
Since the upper bound \iref{Equation sketch-aux-upper} and lower
bound \iref{Equation sketch-aux-lower} do not match as required
in \iref{Equation sketch-cia}, we have to modify $p_N$ to get matching
bounds. Roughly, we take $Q_N:=p_N^{i_N}$ where $i_N$ is the minimal
required power needed to ensure that
\begin{equation*}
  \log \max_{\omega\in\Gamma}\{|Q_N(\omega)|\}\lesssim -N^n.
\end{equation*}
According to \iref{Equation sketch-aux-upper} we need at most $i_N=N^{k+\alpha}$ so that $\operatorname{t}(Q_N)\leq N^{k+1+\alpha}$. Applying the algebraic independence criterion now to $Q_{T^{c_1}},\ldots,Q_{T^{c_2}}$ we see that $\Gamma$ meets no algebraic varieties of dimension $\leq\frac{n}{k+\alpha+1}-1$ and degree and log-height bounded by $T$
whenever \iref{Equation sketch-aux-lower} holds for each $p_N$. For $\alpha$ small enough, Our choice
of $k$ insures that $k<\frac{n}{k+\alpha+1}-1$, and we conclude that $\Gamma(k,T^\alpha,T^\alpha)$ is
contained in the $\operatorname{poly}(T)$ amount of small exceptional balls of \iref{Equation sketch-aux-lower} for $Q_{T^{c_1}},\ldots,Q_{T^{c_2}}$. A small re-indexing argument then allows us to arrange that these balls have size $e^{-T}$ rather than $e^{-T^\theta}$ and to replace $\Gamma(k,T^\alpha,T^\alpha)$ with $\Gamma(k,T,T)$, which finishes the proof.

\begin{Rem}
  In the argument above, the precise lower
  bound \iref{Equation sketch-aux-lower} plays a crucial role. It is useful
  to compare this to the bounds used for the proof of Wilkie's
  conjecture in \so-minimality. In \so-minimality, we do not generally
  have lower bounds in the spirit of
  Corollary~\ref{cor:pointwise-bound}: the axioms of \so-minimality provides upper bounds for numbers of zeros of functions, but not lower bounds for the absolute values of non-vanishing functions. As a first approximation more suitable for comparison with \so-minimality,
  one may consider the bound on the number of zeroes as given in
  Corollary~\ref{cor:zero-bound}. Similar bounds do exist in
  \so-minimality: the number of zeroes of $p_N$ on $\Gamma$ is bounded
  by some polynomial in $N$. However the exponent need not be
  precisely $n$, and a larger exponent would lead to a much weaker
  result (if any) in Theorem~\ref{thm:main}.
\end{Rem}

\section{Examples}

In this section we present a few examples illustrating that as in the
case of Wilkie's conjecture, the $\operatorname{poly}(g,h)$ asymptotic of
Conjectures~\ref{conj:strong} and~\ref{conj:weak} is generally the
optimal asymptotic for the number of unlikely intersections (at least
up to determining the precise degree of the polynomial). We also
present some applications of Theorem~\ref{thm:main} for proving lower
bounds on transcendence degrees of fields generated by exponential
functions.

\subsection{\texorpdfstring{The curve $(z,e^z,\ldots,e^{z^n})$}{The curve (z,exp(z),...,exp(z\^n))}}.

We consider the curve
\begin{equation*}
  \Gamma=\{(z,e^z,\ldots,e^{z^n}) : |z|\leq R\}
\end{equation*}
for some fixed $R$. Denoting the coordinates by $(z,y_1,\ldots,y_n)$ this curve is a
trajectory for the vector field
\begin{equation*}
  \xi:=\frac{\partial}{\partial z} + y_1\frac{\partial}{\partial y_1}+\cdots+nz^{n-1}y_n\frac{\partial}{\partial y_n}.
\end{equation*}
Every invariant (irreducible) variety for this vector field is contained in some
$\{y_j=0\}$, and according to \iref{Proposition d-prop-criterion}
the pair $\xi,\Gamma$ satisfies the arithmetic $D$-property. Whenever
$z\in\mathbb Q$ the corresponding point on $\Gamma$ has
transcendence degree $\leq 1$: more specifically, if $z=\frac{p}q$ then $\Gamma$
intersects the curve
\begin{equation*}
  W_{\frac{p}q} = \{ z=\frac pq, y_2^q=y_1^p,\ldots,y_n^q=y_{n-1}^p \}.
\end{equation*}
We have $\deg W_{\frac pq}=\max(p,q)^{n-1}$ and
$\operatorname{h}(W_{\frac pq})\simeq\log\max(p,q)$. Therefore
\begin{equation*}
  \#\Gamma^{\operatorname{trans}(1)}(1,T,T) \gtrsim T^{\frac2{n-1}}
\end{equation*}
and we see that the counting function indeed grows like some (fractional) power of $T$.

\subsection{\texorpdfstring{The curve $(z,\log z,\log (1+z))$}{The curve (z, log z, log (1+z))}}

We consider the curve
\begin{equation*}
  \Gamma=\{(z,\log z,\log(1+z)) : z\in S\}
\end{equation*}
for some fixed simply-connected compact $S\subsetneq\mathbb C-\{0,-1\}$. Denoting the
coordinates by $(z,y_1,y_2)$ this curve is an image of a trajectory for the vector
field
\begin{equation*}
  \xi:=z(1+z) \big[ \frac{\partial}{\partial z} + \frac1z \frac{\partial}{\partial y_1}+\frac1{1+z}\frac{\partial}{\partial y_2} \big].
\end{equation*}
Once again the arithmetic $D$-property is easy to verify, and
$\Gamma^{\operatorname{trans}(1)}=\Gamma$.

We say $z\in\mathbb C$ is a $T$-algebraic if it satisfies a polynomial of degree and log-height $\leq T$ over $\mathbb Q$. Our goal now will be to count the amount of points in $\Gamma(1,T,T)$ that come a from $T$-algebraic $z$:

\[
    A_T:=\{(z,\log z,\log(1+z))\in\Gamma(1,T,T)):z\text{ is }T\text{-algebraic}\}.
\]

We notice that whenever $z$ is a solution for
\begin{equation*}
  z^n=(1+z)^m
\end{equation*}
then the corresponding point $(z,\log z,\log(1+z))$ belongs to $A_{\max(n,m)}$ as $z$ is algebraic and $n\log z=m\log(1+z)$. Thus, $\#A_T$ is bounded from below by $\operatorname{poly}(T)$. Furthermore, we deduce from Theorem~\ref{thm:main} that $\Gamma(1,T,T)$ is contained in a union of at most a $\operatorname{poly}(T)$ amount of balls of radius $e^{-T^2}$. However, no two points in $A_T$ can lie in the same ball and therefore, $\#A_T$ is also bounded from above by $\operatorname{poly}(T)$. Thus, the estimate of $\operatorname{poly}(T)$ is tight for the amount of points in $\Gamma(1,T,T)$ that come from a $T$-algebraic $z$.

\subsection{Implications of Conjecture~\ref{conj:strong}}
\label{Section conj-trans-applications}

We show that Conjecture~\ref{conj:strong} implies lower bounds for transcendence degrees of fields generated by values of the exponential function that go beyond the current state of the art. This demonstrates that progress towards this conjecture would probably require completely new ideas. We note that this type of deduction is a standard idea in transcendental number theory, and we sketch this here to demonstrate how this area connects to Conjecture~\ref{conj:strong} rather than to present novel ideas.

Let $F\subset\mathbb C$ be a field finitely generated over $\mathbb Q$ and denote its transcendence degree by $q$. Following \cite[V.2]{lang:intro-trans} we fix a transcendence basis $x_1,\ldots,x_q\in F$ and an element $y\in F$ such that 
\[
    F = \mathbb{Q}(x_1,\ldots,x_q,y)
\]
and moreover $y$ is integral over $\mathbb{Z}[x_1,\ldots,x_q]$. For an element $z\in F$ we refer the
reader to \cite[V.2]{lang:intro-trans} for the notion of the \emph{size}\footnote{We note that Lang's notion of word ``size'' in this context differs from the standard usage in Diophantine geometry; we use it here to keep the notation consistent with the referenced results.} of $z$, denoted $\size(z)$. Essentially this measures the degrees and logarithms of the absolute values of the coefficients involved in expressing $z$ as an element in the fraction field of $\mathbb{Z}[x_1,\ldots,x_q,y]$. Lang \cite[V.2~Lemma~1]{lang:intro-trans} proves the following lemma.

\begin{Lem}
    Fix $F,x_1,\ldots,x_q,y$ as above. Then there exists a constant $c>0$
    such that if $\alpha_1,\ldots,\alpha_m\in F$ then
    \begin{align*}
        \size(\alpha_1+\cdots+\alpha_m) &\le c\big(\size(\alpha_1)+\cdots+\size(\alpha_m)\big), \\
        \size(\alpha_1\cdots\alpha_m) &\le c\big(\size(\alpha_1)+\cdots+\size(\alpha_m)\big).
    \end{align*}
\end{Lem}

We also require the following. Extend the notion of size to vectors by taking $\size(\alpha)=\max_i\size(\alpha_i)$.

\begin{Lem}
    Let $\alpha\in(F^*)^n$. Then there exists an algebraic variety 
    $V\subset(\mathbb C^*)^n$ defined over $\mathbb Q$ such that $\alpha\in V$ with $\dim V=q$ and
    \[
        \deg(V)+\operatorname{h}(V) \le C\cdot\size(\alpha)^D
    ,\]
    where $C$ is a constant depending only on $F,x_1,\ldots,x_q,y$ and $D$ depends only on $n,q$.
\end{Lem}

The proof of this lemma is essentially a resultant computation using the definition of size: one expresses each coordinate $\alpha_i$ as a rational function in $x_1,\ldots,x_q,y$ and and then eliminates the variables $x_1,\ldots,x_q,y$ to obtain an algebraic variety $V$ of dimension $q$ in the $\alpha_i$ variables.

We can now state a transcendence consequence of Conjecture~\ref{conj:strong}. Let $\lambda_1,\ldots,\lambda_n\in\mathbb R$ be linearly independent over $\mathbb Q$ and denote by $G_\lambda\subset(\mathbb R^*)^n$ the group
\begin{equation}\label{eq:Gl-def}
    G_\lambda := \{ (e^{\lambda_1 t},\ldots,e^{\lambda_n t}) : t\in\mathbb R\}.
\end{equation}
Denote by $G_\lambda(F)$ the subgroup of points of $G_\lambda$ whose coordinates lie in $F$.
\begin{Prop}\label{prop:trans-application}
    Suppose $q<n-1$. If Conjecture~\ref{conj:strong} holds then the multiplicative rank of $G_\lambda(F)$ is bounded by a constant independent of $F$. Assuming the stronger form given in Remark~\ref{rem:conj-strong-exp} one can take the constant depending only on $n$.
\end{Prop}
\begin{proof}
    The set $X:=G_\lambda\cap(\mathbb R^*)^n$ is a definable subset in the restricted sub-Pfaffian structure to which Conjecture~\ref{conj:strong} applies.
    Suppose $G_\lambda(F)$ contains $r$ multiplicatively independent elements $X_1,\ldots,X_r$. Let $N\in \mathbb N$ and consider the set $S_N\subset G_\lambda(F)$ given by
    \[
        S_N := \{ X_1^{a_1}\cdots X_r^{n_r} : n_1,\ldots,n_r\in\{1,\ldots,N\}\}.
    \]
    According to the lemmas above, each element of $S_N$ is contained in an algebraic variety of dimension $q$ with degree and log-height bounded by $s_N:=C N^{D(n,q)}$ for appropriate constants $C,D$. 

    According to the Ax-Schanuel theorem, $X^{\operatorname{trans}(q)}=X$. Thus using Conjecture~\ref{conj:strong} we have
    \[
        N^r = \# S_N \le X^{\operatorname{trans}(q)}(q,s_N,s_N)\le\operatorname{poly}_{G_\lambda}(s_N)=
        \operatorname{poly}_{G_\lambda}(N^{D(n,q)})
    \]
    and we see that indeed for $r$ large enough (depending only on $G_\lambda$) we get a contradiction for $N\gg1$. Assuming the stronger form from Remark~\ref{rem:conj-strong-exp}, the exponent of the polynomial depends only on the format of $G_\lambda$, that is, on $n$.
\end{proof}

With some minor additional effort one can remove the assumption that $G_\lambda$ is a real group and handle the general complex case.

The conclusion of Proposition~\ref{prop:trans-application} implies that 
if $\{x_i\}\subset\mathbb R$ and $\{y_j\}\subset\mathbb R$ are $\mathbb Q$-linearly independent tuples then the transcendence degree of
\[
    \{ e^{x_i y_j} : i=1,\ldots,r \text{ and } j=1,\ldots,s\}
\]
tends to infinity as $r,s\to\infty$. Indeed, this can be seen by choosing $\{\lambda_i:=x_i\}_i$ and $t=y_j$ in~\eqref{eq:Gl-def}. Establishing such lower bounds for transcendence degrees is currently an open problem in transcendental number theory, and is generally known unconditionally only under some additional technical assumptions on the sets of generators, see \cite[Chapter~14]{NesterenkoPhilippon2001IntroductionToAlgebraicIndependenceTheory} for the best known bounds and \cite[Chapter~14, Definition 2.6]{NesterenkoPhilippon2001IntroductionToAlgebraicIndependenceTheory} for the technical condition. These technical conditions are generally of the type that enables one to repeat the proof above using the weaker Conjecture~\ref{conj:weak} in place of Conjecture~\ref{conj:strong}: they guarantee that the sets $S_N$ constructed above are not only big, but are moreover not contained in a small number of small balls in $G_\lambda$. In particular, making suitable technical assumptions of this sort one can establish a transcendence lower bound using our Theorem~\ref{thm:main} . We note however that these results would certainly be less sharp than the best known ones which are established using more accurate methods specific to group varieties rather than by very general counting principles.

\section{Proof of Theorem~\ref{thm:count-fixed-g}}
\label{Section count-fixed-g}

The idea, quite standard following \cite{pila:blocks}, is to translate
the problem of counting $X^{\operatorname{trans}(k)}(g,h)$ into counting rational
points in a larger-dimensional space depending on $g$. More
specifically, choose some algebraic realization of the space
$\mathcal C^n_{k,g}$ of pure $k$-dimensional subvarieties of
$\mathbb C^n$ with degree at most $g$. For example one may use the Chow
coordinates on this space. Then the following set is $\mathcal S$-definable
\begin{equation*}
  Z := \{ W\in\mathcal C^n_{k,g} : X\cap W\neq\emptyset \}.
\end{equation*}
According to the polylogarithmic counting theorem in the blocks
formulation \cite[Theorem~2]{me:pfaff-wilkie} we have that
\begin{equation*}
  Z(1,h) \subseteq \cup_\eta B_\eta
\end{equation*}
where the union is taken over a collection of $\operatorname{poly}_{X,g}(h)$ blocks
$B_\eta$ of format $O_{X,g}(1)$ and degree $\operatorname{poly}_{X,g}(h)$. It will
suffice to show that for each block $B:=B_\eta$, we have
$\# X_B(k,g,h) = \operatorname{poly}_{X,g}(h)$ where
\begin{equation}\label{Equation X_B-bound}
  X_B:=\{ x\in X^{\operatorname{trans}(k)} : x\in W \text{ for some }W\in B\}.
\end{equation}
First, if for some $W\in B$ the intersection $X\cap W$ has positive
dimension at $x$ then by definition $x\in X^{\operatorname{alg}(k)}$ so such points
are irrelevant for \iref{Equation X_B-bound}. By \so-minimality the
remaining points of $X_B$ lie in the union of the images of a
$\operatorname{poly}_{X,g}(h)$ amount of continuous functions $f_\sigma:S_\sigma\to X$ where
$S_\sigma\subseteq B$ are definable subsets
and both $S_\sigma,f_\sigma$ have format $O_{X,g}(1)$ and degree
$\operatorname{poly}_{X,g}(h)$. By \cite[Proposition~6]{me:pfaff-wilkie} one may further suppose that the $S_\sigma$ are smooth submanifolds. For all $\sigma$ such that $\dim S_\sigma<\dim B$ we
replace $Z$ by $S_\sigma$ and proceed by induction on $\dim Z$ to
obtain the bound
\begin{equation*}
  \# [f_\sigma(S_\sigma)]^{\operatorname{trans}(k)}(g,h) = \operatorname{poly}_{X_,g}(h).
\end{equation*}
Note that $[f_\sigma(S_\sigma)]^{\operatorname{alg}(k)}\subseteq X^{\operatorname{alg}(k)}$ since
$f_\sigma(S_\sigma)\subseteq X$ so we do not miss any points relevant
to \iref{Equation X_B-bound}.

It remains to consider the case $\dim S_\sigma=\dim B$. The set
\begin{equation*}
  \mathrm{LC}_\sigma := \{ f_\sigma(s) : f_\sigma \text{ is locally constant at } s\in S_\sigma\}
\end{equation*}
is finite of size $\operatorname{poly}_{X,g}(h)$ by \so-minimality. We claim
that
$\mathrm{NC}_\sigma:=f_\sigma(S_\sigma)-\mathrm{LC}_\sigma$
satisfies $\mathrm{NC}_\sigma\subseteq X^{\operatorname{alg}(k)}$ so these are again
irrelevant for \iref{Equation X_B-bound}. Indeed, let
$x_0=f_\sigma(s_0)\in\mathrm{NC}_\sigma$. Since $\dim S=\dim B$ and
$B$ is a block, then there is a neighborhood of $s_0$ where $S_\sigma$ is
an open subset of a smooth semialgebraic set. Since $f_\sigma$ is
continuous and not locally constant at $s_0$, it follows that it is
also not locally constant when restricted to a generic semialgebraic
curve $\Gamma$ contained in $S_\sigma$ and passing through $s_0$. Then
the union
\begin{equation*}
  W_0 := \{ x\in W : W\in\Gamma \}
\end{equation*}
has Zariski closure of dimension $k+1$, and intersects $X$ in the
curve $f_\sigma(\Gamma)$ which has positive dimension at
$f_\sigma(s_0)$. Thus $f_\sigma(s_0)\in X^{\operatorname{alg}(k)}$ as claimed.

\section{Transcendence Criteria for Bounded Height and Degree}

In this section our goal is to modify Philippon's classical criterion for algebraic independence (for the original see \cite[Chapter 8, Main criterion for algebraic independence (CIA)]{NesterenkoPhilippon2001IntroductionToAlgebraicIndependenceTheory} and for the modified version see \iref{Proposition modified transcendence criterion}).

\subsection{Definitions}\label{Subsection definitions philippon's criterion}

Let $Q^{\mathrm{proj}} \in \mathbb{C}[X_0, \ldots, X_n]$ be a nonzero homogeneous polynomial, and let $\boldsymbol{x}=(x_0:\ldots,x_n), \boldsymbol{y}=(y_0:\ldots:y_n) \in \mathbb{P}^n(\mathbb{C})$.
\begin{itemize}
\item The \emph{norm of $Q^{\mathrm{proj}}$} is defined by
\[
    \operatorname{norm}(Q^{\mathrm{proj}})
    :=
    \sqrt{\sum_{\alpha} \frac{|Q^{\mathrm{proj}}_\alpha|^2}{\binom{\deg Q^{\mathrm{proj}}}{\alpha}} }
\]
where the sum is over all multi-indices $\alpha = (\alpha_0, \ldots, \alpha_n)$ with $\sum\alpha_i= \deg Q^{\mathrm{proj}}$, and $Q^{\mathrm{proj}}_\alpha$ denotes the corresponding coefficient of the monomial $X^\alpha$.

\medskip

\item The \emph{norm of $Q^{\mathrm{proj}}$ at $\boldsymbol{x}$} is defined by
\[
    \operatorname{norm}(Q^{\mathrm{proj}}, \boldsymbol{x})
    :=
    \frac{ |Q^{\mathrm{proj}}(\boldsymbol{x})| }{ \sqrt{\sum_{i=0}^{n} |x_i|^2 }^{\deg Q^{\mathrm{proj}}} }.
\]

\medskip

\item The \emph{$\operatorname{h_1}$-log-height of $Q^{\mathrm{proj}}$} is defined by
\[
    \operatorname{h}_1(Q^{\mathrm{proj}})
    :=
    \log \left( \operatorname{norm}(Q^{\mathrm{proj}}) \right).
\]

\medskip

\item The \emph{projective distance} between $\boldsymbol{x}$ and $\boldsymbol{y}$ is defined by
\[
    \operatorname{Dist}(\boldsymbol{x}, \boldsymbol{y})
    =
    \frac{ \sqrt{ \sum_{0 \le i < j \le n} |x_i y_j - x_j y_i|^2 } }{ \sqrt{ \sum_{i=0}^n |x_i|^2 } \cdot \sqrt{ \sum_{i=0}^n |y_i|^2 } }.
\]

\medskip

\item For $r > 0$, the \emph{closed projective ball of radius $r$ centered at $\boldsymbol{x}$} is denoted by
\[
    \overline{B_{\boldsymbol{x}, \operatorname{Dist}}(r)}
    =
    \left\{ \boldsymbol{y} \in \mathbb{P}^n(\mathbb{C}) \,\middle|\, \operatorname{Dist}(\boldsymbol{x}, \boldsymbol{y}) \le r \right\}.
\]
\end{itemize}

\subsection{Philippon's criterion}

\begin{ThmCor}[Philippon's Criterion for Algebraic Independence, \mycite{Chapter 8, Main criterion for algebraic independence (CIA)}{NesterenkoPhilippon2001IntroductionToAlgebraicIndependenceTheory}]\label{Theorem Phillipon's criterion}

Let $\boldsymbol{x}\in\mathbb{P}^{n}(\mathbb{C}), k\in\{0,...,n\}$ and $\delta_0,\tau_0,\sigma_0,U_0$ real numbers satisfying $\delta_0\geq1,\sigma_0\geq1,\sigma_0^{k+1}<\tau_0<U_0$. 
Assume that for any $S$ satisfying $\frac{\tau_0}{\sigma_0^{k+1}}<S\le\frac{U_0}{\sigma_0^{k+1}}$, there exists a homogeneous polynomial $Q^{\mathrm{proj}}\in\mathbb Z[X_0,...,X_n]$ satisfying

\begin{enumerate}
    \item $\deg\,Q^{\mathrm{proj}}\le\delta_0$.\label{Condition 1}
    \item $\operatorname{h_1}(Q^{\mathrm{proj}})\le\tau_0$.\label{Condition 2}
    \item $\frac{\operatorname{norm}(Q^{\mathrm{proj}},\boldsymbol{x})}{\operatorname{norm}(Q^{\mathrm{proj}})}\le e^{-S\sigma_0^{k+1}}$.\label{Condition 3}
    \item $Q^{\mathrm{proj}} \;\text{does not have any zero in the closed ball} \;\overline{B_{\boldsymbol{x},\operatorname{Dist}}(e^{-S\sigma_0^{k+2}})}$.\label{Condition 4}
\end{enumerate}

Then, for any projective variety $V$ over $\mathbb Q$ of dimension at most $k$ such that
\begin{enumerate}[resume]
    \item $(k+1)(\operatorname{h}(V)\delta_0+\deg(V)((k+1)\tau_0+3\,\log(n+1)\delta_0))\delta_0^k\sigma_0^{k+1}
        <
    U_0,$\label{Condition assumption V}
\end{enumerate}

we have that $\boldsymbol{x}$ does not lie on $V$.

\end{ThmCor}

\begin{Rem}\label{Remark algebraic independence criterion with dist}

The theorem is actually slightly stronger -- under these conditions, we have that

\[
    \log\operatorname{Dist}(\boldsymbol{x},V)
        \ge
    -U_0
\]

for some appropriate definition of $\operatorname{Dist}(\boldsymbol{x},V)$, see \cite[Chapter 8, Main criterion for algebraic independence (CIA)]{NesterenkoPhilippon2001IntroductionToAlgebraicIndependenceTheory}.

\end{Rem}

\subsection{The modified criterion}\label{Subsection modified criterion}

\begin{Prop}[Modified Criterion for Algebraic Independence]\label{Proposition modified transcendence criterion}
Let $\omega= (\omega_1, \ldots, \omega_n) \in \mathbb{C}^n$.  
Assume there exist integers $a_0, k \in \mathbb{N}$, non-decreasing functions  
\[
\delta, \tau, \sigma, U : \mathbb{N}_{\ge a_0} \to \mathbb{R}_{\ge 2},
\]
and positive constants $\gamma_1, \gamma_2, \gamma_3 \in \mathbb{R}^{>0}$ such that for all integers $a \ge a_0$, the following assumptions are satisfied:

\begin{enumerate}
    \item \label{Assumption tau + ndelta > sigma}$\sigma(a)^{k+1}< \tau(a)+n\delta(a)$.
    \item \label{Assumption growth of U}$a^{\gamma_1} < \frac{U(a)}{(\delta(a) + \tau(a)) \delta(a)^k \sigma(a)^{k+1}} < a^{\gamma_2}$.
    \item \label{Assumption U is polynomial}$U(\lceil a^{\gamma_3} \rceil)\le \tau(a)$.
\end{enumerate}

Let $c_1,c_2,c_3$ be some positive constants satisfying:
\begin{equation} 
c_1<\frac{\gamma_3}{\gamma_2}, \quad c_2>\frac1{\gamma_1}, \quad c_3>1. \label{Equation c_i}
\end{equation}

denote

\[
    A(\omega):= 2n+1 + \log\left(1+c_3^2\cdot n\cdot\max_i\{|\omega_i|^2\}\right).
\]

Then, for any sufficiently large positive real number $T$, the following holds:

Assume that for every integer $N$ satisfying $T^{c_1}\le N\le T^{c_2}$, there exists a polynomial
\[
Q_N \in \mathbb{Z}[x_1, \ldots, x_n]
\]
with degree bounded by $\delta(N)$ and log-height bounded by $\tau(N)$, such that
\begin{equation} \label{Equation assumption QN}
    e^{A(\omega)\cdot(\delta(N) + \tau(N)) - \frac12 U(N-1) \sigma(N)}
        <
    \left| Q_N(\omega) \right| < e^{-U(N)}.
\end{equation}

Then, given any algebraic variety  $V \subseteq \mathbb{C}^n$ defined over $\mathbb Q$ of dimension at most $k$ with
\[
\deg V\le T, \quad \operatorname{h}(V) \le T,
\]
we have that $\omega$ does not lie on $V$.

\end{Prop}

\begin{Rem}

Similarly to \iref{Remark algebraic independence criterion with dist}, one can show that under these conditions:
\[
    \log\operatorname{Dist}(\omega,V)
        \ge
    -U(T^{c_2}).
\]

\end{Rem}

\begin{Rem}

Our definition of $A(\omega)$ is not optimal and can be slightly improved. We work with this definition for simplicity of the calculations (see the proof of \iref{Lemma 2}).

\end{Rem}

\subsection{Proof of the modified criterion}
\begin{proof}[\textbf{Proof of \iref{Proposition modified transcendence criterion}}]
Let us take a positive real number $T$. Denote $a_1:=a_0^{\max\left(1,\frac1{\gamma_3}\right)}$.
We define $i_0$ to be the smallest natural number in $\mathbb N_{\geq a_1}$ which satisfies the following inequality:
\begin{equation}\label{Equation main criterion bound on T}
    2(k+1)T(\delta(i_0)+(k+1)(\tau(i_0)+n\cdot\delta(i_0))+3\,\log(n+1)\delta(i_0))
    <
    \dfrac{U(i_0)}{\delta(i_0)^k\sigma(i_0)^{k+1}}.
\end{equation}
Note that the existence of $i_0$ is ensured by \iref{Assumption growth of U}. We put
\begin{equation} \label{Equation four numbers}
\delta_0:=\delta(i_0), \quad \tau_0:=\tau(i_0)+n\cdot\delta(i_0), \quad \sigma_0:=\sigma(i_0), \quad U_0:=\frac12U(i_0).
\end{equation}
Since the left hand side of \iref{Equation main criterion bound on T} is equal to
\[
    \mathrm{const}\cdot T\cdot\delta(i_0) + \mathrm{const}\cdot T\cdot\tau(i_0)
\]
(for constants that depend only on $k,n$), we have that \iref{Equation main criterion bound on T} implies and is implied (for different constants) by the following inequality:
\[
    \mathrm{const}\cdot T
    <
    \frac{U}{(\delta+\tau)\delta^k\sigma^{k+1}}(i_0).
\]
Combining this inequality with \iref{Assumption growth of U} yields
\begin{equation}\label{Equation bound on i0}
    \frac1{\mathrm{const}}\cdot T^\frac1{\gamma_2}
        \le
    i_0
        \le
    \mathrm{const}\cdot T^{\frac1{\gamma_1}}.
\end{equation}

Let $S$ be any real number in the range $\left(\frac{\tau_0}{\sigma_0^{k+1}},\frac{U_0}{\sigma_0^{k+1}}\right]$. We define $N_0\in (0, i_0]$ to be the smallest natural number such that:
\begin{equation} \label{Equation choice of N0}
    2S
        \le
    \frac{U(N_0)}{\sigma_0^{k+1}}.
\end{equation}
Notice that the choice of $N_0$ implies
\begin{align}
    \tau_0&<\dfrac{U(N_0)}{2}, \label{Equation tau0 and U(N_0)}\\
    \frac{U(N_0-1)}{\sigma_0^{k+1}}
        &<
    2S, \label{Equation U(N_0-1)}
\end{align}
and
\begin{equation} \label{Equation upper bound of N_0}
    N_0
        \le
    i_0
        \le
    \mathrm{const}\cdot T^\frac1{\gamma_1}.
\end{equation}
Using \iref{Assumption U is polynomial} and \iref{Equation tau0 and U(N_0)}, we get
\[
U(\lceil i_0^{\gamma_3}\rceil)\le \tau(i_0)<U(N_0),
\] 
and, since $U$ is an increasing function,
\begin{equation} \label{Equation lower bound of N0}
     \frac1{\mathrm{const}}\cdot T^\frac{\gamma_3}{\gamma_2}
      \le
      i_0^{\gamma_3}
      <
      N_0.
\end{equation}

Using \iref{Equation upper bound of N_0} and \iref{Equation lower bound of N0} together with the choice of $c_1,c_2$, we get that for sufficiently large $T$,
\begin{equation*}
    T^{c_1}
        \le
    N_0
        \le
    T^{c_2}.
\end{equation*}

\medskip

By the statement of the proposition, we assume that there exists a family of polynomials $(Q_N)_{T^{c_1} \le N \le T^{c_2}}$ of degree bounded by $\delta(N)$ and log-height bounded by $\tau(N)$, satisfying \iref{Equation assumption QN}.

Let us show that the quadruple $(\delta_0, \tau_0, \sigma_0, U_0)$ determined by $T$ (see \eqref{Equation four numbers}) satisfies Conditions \eqref{Condition 1}, \eqref{Condition 2}, \eqref{Condition 3} and \eqref{Condition 4}  of \iref{Theorem Phillipon's criterion} for the point $\boldsymbol{\omega}:=(1:\omega)$, with the aid of \iref{Lemma 1} and \iref{Lemma 2}.

\medskip

We keep the notation above and put 
    \[
     Q^{\mathrm{proj}}_N
        :=
    X^{{\mathrm{deg}}\,Q_N}_0Q_N(\frac{X_1}{X_0},\ldots,\frac{X_n}{X_0}).
    \]
Since $\delta$, $\tau$, and $\sigma$ are increasing functions and $N_0 \le i_0$, we have
\begin{align*}
\delta(N_0) &\le \delta_0,\\
\sigma(N_0) &\le \sigma_0,\\
\tau(N_0) &\le \tau(i_0) \le \tau_0,
\end{align*}
so that the degree of $Q^{\mathrm{proj}}_{N_0}$ is bounded by $\delta_0$, and its log-height is bounded by $\tau_0$. 
This shows $Q^{\mathrm{proj}}_{N_0}$ satisfies \iref{Condition 1} of \iref{Theorem Phillipon's criterion}.
We use \iref{Lemma 1} \iref{Part 1 of lemma 1} for $Q^{\mathrm{proj}}_{N_0}$ to get, 
  \[
  {\mathrm{h}}_1(Q^{\mathrm{proj}}_{N_0})\le \tau(i_0)+n\delta(i_0)=\tau_0.
  \]
This shows $Q^{\mathrm{proj}}_{N_0}$ satisfies \iref{Condition 2} of \iref{Theorem Phillipon's criterion}.
Combining \eqref{Equation assumption QN} and \iref{Equation choice of N0}, we see that the polynomial $Q_{N_0}$ satisfies
  \[
  |Q_{N_0}(\omega)|
        <
    e^{-U(N_0)}
        \le
    e^{-\frac12 U(N_0)-S\sigma_0^{k+1}}.
  \]
The above inequality together with \iref{Lemma 1} \iref{Part 2 of lemma 1} ensures that as long as $T$ is sufficiently large, 
  \[
  \dfrac{{\mathrm{norm}}(Q^{\mathrm{proj}}_{N_0},\boldsymbol{\omega})}{{\mathrm{norm}}(Q^{\mathrm{proj}}_{N_0})}\le  e^{-\frac12 U(N_0)-S\sigma_0^{k+1}}e^{\frac{\log(n+1)}2\delta(N_0)}\le e^{-S\sigma^{k+1}_0}.
  \]
  
This shows $Q^{\mathrm{proj}}_{N_0}$ satisfies \iref{Condition 3} of \iref{Theorem Phillipon's criterion}.
The lower bound of \iref{Equation assumption QN} for $Q_{N_0}$ together with \iref{Equation U(N_0-1)} implies
\[
    e^{A(\omega)\cdot(\delta(N_0)+\tau(N_0))-S\sigma_0^{k+2}}
    \le 
    e^{A(\omega)\cdot(\delta(N_0)+\tau(N_0))-S\sigma_0^{k+1}\sigma(N_0)}
    <
    |Q_{N_0}(\omega)|.   
\]
Finally the choice of $A(\omega)$ and \iref{Lemma 2} for $r:=e^{-S\sigma^{k+2}_0}$ and $d:=c_3$ imply the polynomial $Q^{\mathrm{proj}}_{N_0}$ satisfies \iref{Condition 4} of \iref{Theorem Phillipon's criterion} for sufficiently large $T$.

\medskip 
 
Let $V$ be an affine variety defined over $\mathbb Q$, of dimension at most $k$, satisfying
\[
\dim V \le T, \quad \mathrm{h}(V) \le T.
\]
Applying \iref{Theorem Phillipon's criterion} to the quadruple $(\delta_0, \tau_0, \sigma_0, U_0)$ and the sequence of polynomials $(Q^{\mathrm{proj}}_N)_N$, we observe that, in view of \iref{Equation main criterion bound on T}, the quadruple $(\delta_0, \tau_0, \sigma_0, U_0)$ satisfies \iref{Condition assumption V} of \iref{Theorem Phillipon's criterion} for the projective closure of $V$ in $\mathbb{P}^n(\mathbb{C})$, and thus, the conclusion follows.
\end{proof}

\section{Construction of Small Polynomials from Low Transcendence Points}

\subsection{Statement of the proposition}

The aim of this section is to prove the following proposition.

\begin{Prop}\label{Proposition small polynomial}
Let $k,n\in \mathbb N$ with $k\le n-2$ and let $c\in\mathbb R_{>0}$ and $\alpha\in(0,n-k-1)$ be some constants. Then, for any sufficiently large integer $N$, the following holds.

Assume that there exists a collection of different points $\omega_1,...,\omega_{\lceil N^{n-k-\alpha}\rceil}\in\Gamma$ such that each $\omega_i$ belongs to a prime ideal $\mathfrak p_i\subseteq\mathbb Q[x_1,\ldots,x_n]$ of dimension $k$ and with $\operatorname{t}(\mathfrak p_i)\le c\cdot N^\alpha$. Then, there exists a non-zero polynomial $p\in\mathbb Z[x_1,\ldots,x_n]$ that satisfies:
\begin{align}
    \operatorname{t}(p)&\le N, \label{tp le}\\
    \max_{\omega\in \Gamma}\{|p(\omega)|\}&\le e^{-\frac1C\cdot N^{n-k-\alpha}}. \label{|p(omega)|}
\end{align}
where $C$ is a constant depending only on $\Gamma$ and on $c,\alpha$.

\end{Prop} 
\subsection{Auxiliary results}
To prove \iref{Proposition small polynomial}, we first prepare several auxiliary results.
The upcoming proposition ensures the existence of a non-zero homogeneous polynomial $p$ with moderately bounded $\operatorname{t}(p)$ contained in a given family of homogeneous prime ideals of fixed dimension defined over a number field. This result relates to Bound~\eqref{tp le} in \iref{Proposition small polynomial} and is achieved by using the following:

\begin{ThmCor}[\mycite{Theorem~3}{Nesterenko1985EstimatesForTheCharacteristicFunctionOfAPrimeIdeal}]\label{Theorem nesterenko bound on amount of residue polynomials}

Let $\mathfrak p\subseteq\mathbb Z[X_0,...,X_n] $ ($n\geq1$) be a projective, $k$-dimensional, prime ideal with $\mathfrak p\cap\mathbb Z=(0)$, let $T\in\mathbb R^{>0}$ and let $t\in\mathbb N$ such that:

\[
    e^{(k+1)t}
        \geq
        \max(30,2k\deg(\mathfrak p)t+3\deg(\mathfrak p)\log\operatorname{H}(\mathfrak p)+6n(k+1)\deg(\mathfrak p)^2+3(k+1)\deg(\mathfrak p)^3).
\]

Then,

\[
    \log|\mathfrak M_\mathfrak p(T,t)|
        \leq
    6^{k+1}((n^2+n+2)\deg(\mathfrak p)t^{k+1}+\log\operatorname{H}(\mathfrak p)t^{k+1}+\deg(\mathfrak p)t^k\log T),
\]

where $\mathfrak{M}_{\mathfrak{p}}(T,t)$ denote the set of residues modulo $\mathfrak{p}$ of homogeneous polynomials in $\mathbb{Z}[X_0, \ldots, X_n]$ of degree $t$ and Height at most $T$.

\end{ThmCor}

\begin{Prop}\label{Proposition polynomial in intersection of ideals}
Let $\{\mathfrak p_1,\ldots,\mathfrak p_s\}$ be a collection of prime ideals in $\mathbb Q[x_1,....,x_n]$ of dimension $k$.
Then, there exists a non-zero homogeneous polynomial $p\in\mathbb Z[x_1,...,x_n]\cap\bigcap_i\mathfrak p_i$ satisfying
\[
    \operatorname{t}(p)\le C\cdot\left(\sum_i\operatorname{t}(\mathfrak p_i)\right)^\frac1{n-k}
\]

for some constant $C$ depending only on $n$.
\end{Prop}

\begin{proof}

Let $t$ be a natural number with $t \ge 30$, and set $T := e^t$. Assume that
\[
    \forall i:\quad t \ge \mathrm{const} \cdot \operatorname{t}(\mathfrak{p}_i)^{\frac1n},
\]
for a sufficiently large constant. We apply \iref{Theorem nesterenko bound on amount of residue polynomials} to each $\mathfrak{q}_i$ which is defined as the projectivization of each $\mathfrak{p}_i \cap \mathbb{Z}[x_1, \ldots, x_n]$, with our parameters $T$ and $t$. We verify the growth condition required for the application of the theorem (notice that $\deg(\mathfrak{q}_i)=\deg(\mathfrak{p}_i)$ and $\operatorname{h}(\mathfrak{q}_i)=\operatorname{h}(\mathfrak{p}_i)$):
\[
    e^{(k+1)t} \ge e^t \ge \frac{t^{3n+1}}{(3n+1)!}
\]
dominates the right-hand side:
\[
    2k\deg(\mathfrak{p}_i)t + 3\deg(\mathfrak{p}_i)\operatorname{h}(\mathfrak{p}_i) 
    + 6n(k+1) \deg(\mathfrak{p}_i)^2 + 3(k+1)\deg(\mathfrak{p}_i)^3 
    \le \mathrm{const} \cdot \operatorname{t}(\mathfrak{p}_i)^3 t.
\]
Hence, the condition is satisfied under our assumption on $t$. Thus, we obtain the bound
\begin{align*}
    \forall i:&\\
    &\log \left| \mathfrak{M}_{\mathfrak{q}_i}(T,t) \right|
    \le 6^{k+1} \left( (n^2 + n + 2) \deg(\mathfrak{q}_i) t^{k+1} 
        + \operatorname{h}(\mathfrak{p}_i) t^{k+1} 
        + \deg(\mathfrak{p}_i)t^k \log T \right)
\end{align*}

which gives us the estimate
\begin{equation} \label{Equation bound of M(T,t)}
    \log \left| \mathfrak{M}_{\mathfrak{q}_i}(T, t) \right| 
    \le \mathrm{const} \cdot \operatorname{t}(\mathfrak{p}_i) \, t^{k+1}.
\end{equation}
We now aim to find two distinct homogeneous polynomials $p_1, p_2 \in \mathbb{Z}[X_0, \ldots, X_n]$ with $\operatorname{t}(p_1), \operatorname{t}(p_2) \le t$ such that
\[
    \forall i: \quad p_1 \equiv p_2 \mod \mathfrak{q}_i.
\]
Let us denote by $\mathscr{S}(t)$ the set of homogeneous polynomials $p\in \mathbb{Z}[X_0, \ldots, X_n]$ with $\operatorname{t}(p)\le t$:
\[
    \mathscr{S}(t) = \left\{ p \in \mathbb{Z}[X_0, \ldots, X_n] \mid p: \text{ homogeneous, } \operatorname{t}(p) \le t \right\}.
\]
Consider the natural map between finite sets:
\[
    \mathscr{S}(t) \longrightarrow \prod_i \mathfrak{M}_{\mathfrak{q}_i}(T, t), \quad p \mapsto \left(...,p \bmod \mathfrak{q}_i,...\right)_i.
\]
By the pigeonhole principle, if
\[
    |\mathscr{S}(t)| > \prod_i \left| \mathfrak{M}_{\mathfrak{q}_i}(T, t) \right|,
\]
then such $p_1$ and $p_2$ exist. Now, we estimate the cardinality of $\mathscr{S}(t)$:

\[
    |\mathscr{S}(t)| \ge (1 + \lfloor T \rfloor)^{\left( 1 + \left\lfloor \frac{t}{n} \right\rfloor \right)^n} \ge T^{(\frac tn)^n} = e^{\frac{t^{n+1}}{n^n}}.
\]
On the other hand, thanks to \iref{Equation bound of M(T,t)}, 
\[
    \prod_i \left| \mathfrak{M}_{\mathfrak{q}_i}(T, t) \right| 
    \le e^{\mathrm{const} \cdot t^{k+1} \sum_i \operatorname{t}(\mathfrak{p}_i)}.
\]
Therefore, it suffices that
\[
    \frac{t^{n+1}}{n^n} > \mathrm{const} \cdot t^{k+1} \sum_i \operatorname{t}(\mathfrak{p}_i),
\]
which reduces to
\[
    t^{n - k} > \mathrm{const} \cdot \sum_i \operatorname{t}(\mathfrak{p}_i).
\]
This inequality holds whenever
\[
    t \ge \mathrm{const} \cdot \left( \sum_i \operatorname{t}(\mathfrak{p}_i) \right)^{\frac1{n-k}}.
\]

In conclusion, the following set of conditions suffices:
\begin{align*}
    t &\ge \mathrm{const} \cdot \left( \sum_i \operatorname{t}(\mathfrak{p}_i) \right)^{\frac1{n-k}}, \\
    t &\ge \mathrm{const} \cdot \max_i \left\{ \operatorname{t}(\mathfrak{p}_i)^{\frac1n} \right\}, \\
    t &\ge 30.
\end{align*}
All of these are simultaneously satisfied by choosing
\[
    t \ge \mathrm{const} \cdot \left( \sum_i \operatorname{t}(\mathfrak{p}_i) \right)^{\frac1{n-k}},
\]
with a sufficiently large constant. Once such polynomials $p_1$ and $p_2$ are found, the affinization of their difference $p(x_1,...,x_n) := (p_1 - p_2)(1,x_1,...,x_n)$ lies in $\mathbb Z[x_1,...,x_n]\cap\bigcap_i \mathfrak{p}_i$ and satisfies the required bounds.
\end{proof}

The upcoming proposition estimates the maximum value on $\Gamma$ of a polynomial that has sufficiently many zeroes on $\Gamma$. This is related to Bound \eqref{|p(omega)|} in \iref{Proposition small polynomial} and uses the following:

\begin{ThmCor}[\mycite{Lemma~1}{IlyashenkoYakovenko1996CountingRealzerosOfAnalyticFunctionsSatisfyingLinearOrdinaryDifferentialEquations}]\label{Theorem ilyashenko yakovenko bound on number of zeroes of analytic function}

Let $f:\overline{D_1}\to\mathbb C$ be analytic on $D_1$ and continuous on $\overline{D_1}$. Then:

\[
    N_{\overline{D_{\frac1e}}}(f)
        \leq
    C\cdot\log\frac{\displaystyle\max_{z\in \overline{D_1}} |f(z)|}{\displaystyle\max_{z\in \overline{D_{\frac1e}}} |f(z)|},
\]

where $\operatorname{N}_{\overline{D_{\frac1e}}}(f)$ denotes the number of zeroes of $f$ in $\overline{D_{\frac1e}}$ and $C$ is some absolute constant.

\end{ThmCor}
 
\begin{Prop}\label{Proposition polynomial with many zeroes is small}
There exists a constant $C$ depending only on $\Gamma$ such that any non-zero polynomial $p \in \overline{\mathbb{Q}}[x_1, \ldots, x_n]$ with at least $s$ distinct zeroes in $\Gamma$ satisfies the following inequality:
\[
    \max_{\omega \in \Gamma}\{|p(\omega)|\} \le
    e^{C\cdot \operatorname{t}(p) - \frac1C s}.
\]
\end{Prop}

\begin{proof}

We apply \iref{Theorem ilyashenko yakovenko bound on number of zeroes of analytic function} to the composition $f := p \circ \gamma$. Since we assume that $p$ has at least $s$ distinct zeroes on $\Gamma$, we obtain the lower bound:

\[
    e^{\frac1{\mathrm{const}} s}
    \le
    \frac{\displaystyle\max_{z\in\overline{D_1}} |p(\gamma(z))|}{\displaystyle\max_{z\in\overline{D_\frac1e}} |p(\gamma(z))|}.
\]
Moreover, observe that
\[
    \max_{\omega \in\overline{D_1}} |p(\gamma(z))|
    \le
    e^{\mathrm{const} \cdot \operatorname{t}(p)}.
\]
Combining these estimates, we conclude that
\[
    \max_{\omega \in \Gamma}\{|p(\omega)|\}
    \le
    e^{\mathrm{const} \cdot \operatorname{t}(p) - \frac1{\mathrm{const}} s},
\]
as desired.

\end{proof}

\subsection{Proof of Proposition~\ref{Proposition small polynomial}}
\begin{proof}
Let $N$ be a positive natural number. Assume that there exists a collection of different points $\omega_1,...,\omega_{\lceil N^{n-k-\alpha}\rceil}\in\Gamma$ such that each $\omega_i$ belongs to a prime ideal $\mathfrak p_i$ of dimension $k$ and with $\operatorname{t}(\mathfrak p_i)\le c\cdot N^\alpha$.
By \iref{Proposition polynomial in intersection of ideals} applied to the $\mathfrak p_i$, there exists a non-zero polynomial $p \in\mathbb Z[x_1,..,x_n]\cap\bigcap_i\mathfrak{p}_i$ such that
\begin{equation} \label{Equation bound of t(p)}
    \operatorname{t}(p) \le \mathrm{const} \cdot N.
\end{equation}
In particular, we have
\[
  p(\omega_i) = 0 \;\;\text{for all} \;i.
\]
Applying \iref{Proposition polynomial with many zeroes is small} together with \iref{Equation bound of t(p)}, we obtain
\begin{align*}
    \max_{\omega \in \Gamma}\{|p(\omega)|\}&\le e^{\mathrm{const} \cdot N - \frac1{\mathrm{const}} N^{n-k-\alpha}}.
\end{align*}
Since $n-k-\alpha>1$, it follows that for sufficiently large $N$,
\[
     \max_{\omega \in \Gamma}\{|p(\omega)|\} \le e^{-\frac1{\mathrm{const}} \cdot N^{n-k-\alpha}}.
\]
This completes the proof of the assertion.
\end{proof}

\section{The Main Theorem}

In this section, we assume that $k<\sqrt n-1$.

We start by defining some parameters to be used in the rest of this section.

\subsection{Choice of parameters}\label{Subsection choice of parameters}

Choose $\alpha,\varepsilon,\eta,\gamma_1,\gamma_2,\gamma_3,\theta$ (constants) to be positive real numbers satisfying

\[
\begin{alignedat}{3}
&\alpha               &\:<\:&   \frac{n}{k+1}-(k+1),\\
&\varepsilon   &\:<\:&   n-(k+1)(k+1+\alpha),\\ 
&\eta                   &\:<\:&   \min\left(\varepsilon,\dfrac{n-\varepsilon}{k+1}\right),\\
&\gamma_2          &\:=\:&   \varepsilon-\eta,\\
&\gamma_1          &\:<\:&   \gamma_2,\\
&\gamma_3          &\:<\:&   \dfrac1{k+1}-\dfrac\varepsilon{n(k+1)},\\
&\theta              &\:<\:&    \frac{\eta}{k+1}.
\end{alignedat}
\]

We define functions $\alpha,\tau,\sigma,U:\mathbb{N}\longrightarrow \mathbb R$ by
\begin{align*}
    \delta(x)=\tau(x)&=x^\frac{n-\varepsilon}{k+1},\\
    \sigma(x)&=x^\frac{\eta}{k+1},\\
             U(x)&=x^n.
\end{align*}

Choose constants $c_1,c_2,c_3$ and use the function $A(\omega)$ as in \iref{Equation c_i}, so that:

\begin{equation} 
c_1<\frac{\gamma_3}{\gamma_2}, \quad c_2>\frac1{\gamma_1}, c_3>1,\quad A(\omega)= 2n+1 + \log\left(1+c_3^2\cdot n\cdot\max_i\{|\omega_i|^2\}\right).
\end{equation}

\subsection{Statement of the main theorem}

\begin{ThmCor}[the main theorem] \label{Theorem main theorem}
Let $k,n$ be natural numbers with $n\ge 2$ and $k<\sqrt n - 1$. We denote by $\Gamma(T)$ the set of points where $\Gamma$ intersects any $k$-dimensional algebraic variety $V$ defined over $\mathbb Q$ with $\operatorname{t}(V)\le T$. Then, for any sufficiently large positive integer $T$, the points in $\Gamma(T)$ are contained in the images (under $\gamma$) of a polynomial-in-$T$ amount of disks of exponentially small radius.

More explicitly, for sufficiently large $T$, we have that $\Gamma(T)$ is contained in the images of a

\[
    C\cdot T^{c_2(n-k-\alpha+1)\max\left(1,\frac1{c_1\cdot\alpha}\right)}
\]
amount of disks of sum of diameters less than

\[
    e^{-\frac1C\cdot T^{\theta\cdot\max(c_1,\frac1{\alpha})}}
\]

for some constant $C$ depending only on $\xi,\gamma,K$ and on the choice of the constants in \iref{Subsection choice of parameters}.
\end{ThmCor}

\begin{Rem}\label{Remark main theorem}

Assuming $T\geq 2$, since $\Gamma(T)\subseteq \Gamma(T^a)$ for any constant $a\geq 1$, then, by covering $T^a$ instead of $a$, then, for any $\rho>0$, we can can actually cover  $\Gamma(T)$ using a
\[
C\cdot\max\left(T^{c_2(n-k-\alpha+1)\max\left(1,\frac1{c_1\cdot\alpha}\right)},T^{\frac{c_2(n-k-\alpha+1)}{{\theta\cdot c_1}}\rho}\right)
\]
amount of disks of sum of diameters less than
\[
    e^{-T^\rho},
\]
for some constant $C$ depending only on $\xi,\gamma,K$ and on the choice of the constants in \iref{Subsection choice of parameters}.

\end{Rem}

\subsection{Relevant theorems from other works}

This section relies on \iref{Theorem order-bound} (we recall here the definition $\operatorname*{ord}_\Gamma p:=-\log\max_{\omega\in\Gamma}\{|p(\omega)|\}$), together with the following lemmas:
  
\begin{Lem}[See \mycite{Remark 8.2.4}{BinyaminiSalant2026Technical}]\label{Lemma lower bound on polynomial with disks}

There exists a constant $C>0$ depending only on $\Gamma$ that satisfies the following:

For any non-zero polynomial $p\in\mathbb Z[x_1,...,x_n]$ satisfying $(-\log\max_{\omega\in\Gamma}\{|p(\omega)|\})\geq 0$, given any $0<h<\frac1e$, there are a finite number of disks $\{D^i\}_{i=1}^a$ with sum of diameters less than $h$, and with $a\leq C\cdot(-\log\max_{\omega\in\Gamma}\{|p(\omega)|\}+\operatorname{t}(p))$, such that:

\begin{align*}
    \forall x\in\overline{D_\frac1e}-\bigcup_i D^i&:\\
    \log|p(\gamma(x))|
        &\geq
        -C\cdot\log\frac1h\cdot( 
    -\log\max_{\omega\in\Gamma}\{|p(\omega)|\}
        +
    \operatorname{t}(p)).
\end{align*}

\end{Lem}

\begin{Lem}[see \mycite{Proposition 4.7}{NesterenkoPhilippon2001IntroductionToAlgebraicIndependenceTheory}]\label{Lemma t of prime ideal containing ideal}

Let $W$ be an equidimensional variety defined over $\mathbb Q$, not necessarily irreducible, and let $\omega\in W$. Then, $W$ contains a component $V$, irreducible over $\mathbb Q$, which contains $\omega$ and satisfies:

\[
    \operatorname{t}(V)\leq (n^2+1)\cdot\operatorname{t}(W).
\]

\end{Lem}

\subsection{Proving the main theorem}

\subsubsection{Framework of the proof}

Notice, a direct computation yields that the functions $\delta,\tau,\sigma,U$ satisfy Assumptions~\eqref{Assumption tau + ndelta > sigma},~\eqref{Assumption growth of U},~\eqref{Assumption U is polynomial} in \iref{Proposition modified transcendence criterion} for large enough $a$.

Let $T$ be a positive integer. We define $R:=\max\left(T,T^{\frac1{c_1\cdot\alpha}}\right)$. If the size of $\Gamma(T)$ less than $T^{c_2(n-k-\alpha)\max\left(1,\frac1{c_1\cdot\alpha}\right)}$, then, $\Gamma(T)$ can be covered by $T^{c_2(n-k-\alpha)\max\left(1,\frac1{c_1\cdot\alpha}\right)}$ balls of any size, and we are done. So, we assume otherwise and work with $R$.

Let $N \in \mathbb{N} \cap[R^{c_1}, R^{c_2}]$. We see by \iref{Lemma t of prime ideal containing ideal} that $\Gamma(T)$ gives us a collection of at least $N^{n-k-\alpha}$ distinct points $\omega_i$, each of which belongs to the zero locus of a prime ideal $\mathfrak{p}_i\subseteq\mathbb Q[x_1,...,x_n]$ of dimension $k$ and with $\operatorname{t}(\mathfrak{p}_i) \le (n^2+1)\cdot N^\alpha$.

In the next parts, we will use these points to prove that if $T$ is large enough, then $\Gamma(R)$ can be covered by the images under $\gamma$ of a $C\cdot R^{c_2\cdot\left(n-k-\alpha+1\right)}$ amount of disks with sum of diameters less than $e^{-\frac1C\cdot R^{c_1\theta}}$, for some constant $C$ depending only on $\xi,\gamma,K$ and on the choice of the constants in \iref{Subsection choice of parameters}. Since $R=\max\left(T,T^{\frac1{c_1\cdot\alpha}}\right)$, the theorem follows.

\subsubsection{Construction of auxiliary polynomials}

\begin{Prop} \label{final prop}

Assuming that $k\leq n-2$ and that $T$ is sufficiently large, then, for each integer $N \in [R^{c_1}, R^{c_2}]$, there exists a non-zero polynomial 
\[
    Q_N \in\mathbb Z[x_1, \ldots, x_n]
\]
satisfying the following properties:
\begin{enumerate}

\item $\mathrm{t}( Q_N) \le C\cdot N^{k+1+\alpha}.$\label{Property first property}
\item $
    e^{-N^{n+\theta}} 
    < 
    | Q_N(\omega)| 
    < 
    e^{-N^n}
    \quad \text{for all } \omega \in \gamma\left( \overline{D_{\frac1e}} - A_N \right).$\label{Property second property}
\end{enumerate}
where $A_N $ is a union of at most $CN^{n-k-\alpha}$ disks with sum of diameters less than $e^{-\frac1C N^\theta}$ for some constant $C$ depending only on $\xi,\gamma,K$ and on the choice of the constants in \iref{Subsection choice of parameters}.

\end{Prop}

\begin{proof}
By \iref{Proposition small polynomial}, there exists a non-zero polynomial $p_N\in\mathbb Z[x_1,...,x_n]$ satisfying the following bounds:

\begin{equation}\label{Equation bound_tpN}
    \mathrm{t}(p_N)
        \le
    N,
\end{equation}

\begin{equation*}
    \max_{\omega \in \Gamma}\{|p_N(\omega)|\}
        \le 
    e^{-\frac1{\mathrm{const}} \cdot N^{n-k-\alpha}}. \nonumber
\end{equation*}

By \iref{Lemma lower bound on polynomial with disks} applied to $p_N$ and $h$ to be determined, we find a union $A_N$ of at most $\mathrm{const} \cdot N^{n-k-\alpha}$ disks with sum of diameters at most $h$ such that for all $\omega \in \gamma\left( \overline{D_{\frac1e}} - A_N \right)$ we have:
\begin{equation} \label{Equation lower bound pN}
    |p_N(\omega)|
        >
    e^{\mathrm{const}\cdot\log\frac1h\cdot\log\max_{\omega\in\Gamma}\{|p_N(\omega)|\}-\mathrm{const}\cdot\log\frac1h\cdot\operatorname{t}(p_N)}.
\end{equation}

\medskip

Let $i_N \in \mathbb{N}$ be the smallest positive integer such that
\begin{equation}\label{Equation upper bound pN^iN}
    \max_{\omega \in \Gamma}\{|p_N^{i_N}(\omega)|\}< e^{-N^n},
\end{equation}
or equivalently,
\[
i_N\cdot\log\max_{\omega \in \Gamma}\{|p_N(\omega)|\}< -N^n.
\]
Note that
\begin{equation} \label{Equation bound of iN}
     i_N \le {\mathrm{const}} \cdot N^{k+\alpha}.
\end{equation}
and that
\begin{equation}\label{Equation iN-1}
(i_N-1)\cdot\log\max_{\omega \in \Gamma}\{|p_N(\omega)|\}
    \ge
-N^n.
\end{equation}
Set
\[
     Q_N := p_N^{i_N}.
\]

From \iref{Equation bound_tpN} and \iref{Equation bound of iN} we deduce
\begin{equation}
    {\mathrm t}(Q_N) 
        \le
    {\mathrm{const}} \cdot N^{k+1+\alpha},
    \label{Equation bound_tQN} 
\end{equation}

\medskip

which is \iref{Property first property}.

We now verify that $Q_N$ satisfies \iref{Property second property}.

We established the upper bound in \iref{Equation upper bound pN^iN}. As for the lower bound, we have that from the above Inequalities \eqref{Equation lower bound pN}, \eqref{Equation iN-1}, \eqref{Equation bound_tpN}, and \eqref{Equation bound of iN}, that for all $\omega \in \gamma\left( \overline{D_{\frac1e}} - A_N \right)$:
\begin{align*}
|Q_N(\omega)|
    &=
|p_N^{i_N}(\omega)|
    \\&>
e^{\mathrm{const}\cdot\log\frac1h\cdot\log\max_{\omega\in\Gamma}\{|p_N(\omega)|\}\cdot(i_N-1)+\mathrm{const}\cdot\log\frac1h\cdot\log\max_{\omega\in\Gamma}\{|p_N(\omega)|\}-\mathrm{const}\cdot\log\frac1h\cdot\operatorname{t}(p_N)\cdot i_N}
    \\&\ge
e^{-\mathrm{const}\cdot\log\frac1h\cdot N^n+\mathrm{const}\cdot\log\frac1h\cdot\log\max_{\omega\in\Gamma}\{|p_N(\omega)|\}-\mathrm{const}\cdot\log\frac1h\cdot N\cdot N^{k+\alpha}},
\end{align*}
and combining this with \iref{Theorem order-bound} applied to $p_N$ and \iref{Equation bound_tpN}, we get that:
\[
    |Q_N(\omega)|
        \ge
    e^{-\mathrm{const}\cdot\log\frac1h\cdot(2N^n+N^{k+1+\alpha})},
\]
and since $n\ge k+1+\alpha$, then we finally have that for all $\omega \in \gamma\left( \overline{D_{\frac1e}} - A_N \right)$:
\[
    |Q_N(\omega)|
        >
    e^{-\mathrm{const}\cdot\log\frac1h\cdot N^n},
\]
and by a choice of $h:=e^{-\frac1C N^\theta}$ for some large enough constant $C$, then we have that
\[
    |Q_N(\omega)|
        >
    e^{-N^{n+\theta}},
\]
which completes the proof.

\end{proof}

\subsubsection{Completing the proof}
\begin{proof}[\textbf{Proof of \iref{Theorem main theorem}}]

Let

\[
    \bigl( Q_N \bigr)_{\,{R^{c_1} \le N \le R^{c_2}}},
    (A_N)_{\,R^{c_1} \le N \le R^{c_2}}
\]

be the families of polynomials and disks obtained in Proposition~\ref{final prop}, and let
\[
    \omega = (\omega_1,\ldots,\omega_n) 
    \in \gamma\!\left( 
        \overline{D_{\frac1e}} - 
        \bigcup_{N \in \mathbb{N} \cap [R^{c_1},\,R^{c_2}]} A_N
    \right).
\]

We remark that for $R$ large enough, we get that $\bigcup_{N \in \mathbb{N} \cap [R^{c_1},\,R^{c_2}]} A_N$ is a union of at most a $\mathrm{const}\cdot R^{c_2\cdot\left(n-k-\alpha+1\right)}$ amount of disks with sum of diameters less than $e^{-\frac1{\mathrm{const}}\cdot R^{c_1\theta}}$.

By our assumption that $k<\sqrt n-1$ and by our choice of $\varepsilon$, we have that for $N$ (or equivalently $T$) sufficiently large:
\[
    \operatorname{t}(Q_N)
        \leq
    N^{\frac{ n- \varepsilon }{k+1}}
        =
    \delta(N)
        =
    \tau(N).
\]

We show that these polynomials $Q_N$ satisfy \iref{Equation assumption QN} in \iref{Proposition modified transcendence criterion} for $\omega$, with the functions $(\delta,\tau,\sigma,U)$ and with $A(\omega)$ as before. 

Since

\[
    A(\omega)\cdot\bigl( \delta(N) + \tau(N) \bigr) 
    - \frac12\, U(N-1)\,\sigma(N)
        \le
    2\max_{\omega\in\Gamma}A(\omega)\cdot N^{\frac{n}{k+1} - \frac{\varepsilon}{k+1}}
       -\frac12\cdot (N-1)^{n + \frac{\eta}{k+1}},
\]

then, in order to prove \iref{Equation assumption QN} for $N$ sufficiently large, 
it suffices to show that
\[
    \max\!\left\{
        \frac{n}{k+1} - \frac{\varepsilon}{k+1},\,
        n + \theta
    \right\}
    < 
    n + \frac{\eta}{k+1}.
\]
This follows immediately from the definition of $\theta$.

Thus, by \iref{Proposition modified transcendence criterion}, $\omega$ does not belong to an algebraic variety of dimension $\le k$ and of degree and log-height $\le R$.
And so, we have that all the points in $\Gamma(R)$ must belong to $\gamma\left(\bigcup_{N\in\mathbb N\cap[R^{c_1},R^{c_2}]}A_N\right)$, which completes the proof of \iref{Theorem main theorem}.
\end{proof}

\section{Appendix: Calculations Relating to \texorpdfstring{Definitions \iref{Subsection definitions philippon's criterion}}{Philippon's Criterion Definitions}}

\begin{Lem}\label{Lemma 1}
Let $\delta, \tau \geq 2$ be real numbers, and let $Q \in \mathbb{Z}[x_1, \ldots, x_n]$ be a non-zero polynomial of degree at most $\delta$, with log-height at most $\tau$.
Define the homogeneous polynomial
\[
Q^{\mathrm{proj}} := X_0^{\deg Q} \cdot Q\left(\frac{X_1}{X_0}, \ldots, \frac{X_n}{X_0}\right).
\]
Then the following estimates hold:

\begin{enumerate}
\item $\operatorname{h}_1(Q^{\mathrm{proj}}) \le \tau + n\log\delta$.\label{Part 1 of lemma 1}

\item Let $x \in \mathbb{C}^n$, and write $\boldsymbol{x} = (1 : x) \in \mathbb{P}^n(\mathbb{C})$. Then,\label{Part 2 of lemma 1}

\[
\frac{\operatorname{norm}(Q^{\mathrm{proj}}, \boldsymbol{x})}{\operatorname{norm}(Q^{\mathrm{proj}})} 
\le 
|Q(x)| \cdot e^{\frac{\log(n+1)}2\delta}.
\]
\end{enumerate}

\end{Lem}

\begin{proof}
\hfill\\
\begin{enumerate}
\item The definition of $\operatorname{h}_1(Q^{\mathrm{proj}})$ implies
\[
\operatorname{h}_1(Q^{\mathrm{proj}}) 
\le \operatorname{h}(Q) + \frac12 \,\log\left((\deg Q^{\mathrm{proj}} + 1)^n\right)
\le \tau + n \,\log\,\delta.
\]
\item The definitions of the norms ensure
\[
\frac{\operatorname{norm}(Q^{\mathrm{proj}}, \boldsymbol{x})}{\operatorname{norm}(Q^{\mathrm{proj}})}
\le \frac{|Q(x)|}{\sqrt{\sum_{\alpha} \frac{|Q^{\mathrm{proj}}_\alpha|^2}{\binom{\deg Q^{\mathrm{proj}}}{\alpha}}}}.
\]
Since the binomial coefficients satisfy $\binom{\deg Q^{\mathrm{proj}}}{\alpha} \le (n+1)^{\deg Q^{\mathrm{proj}}}$ for multi-indices $\alpha$, the denominator is at least $e^{-\frac{\log(n+1)}2\delta}$. Thus,
\[
\frac{\operatorname{norm}(Q^{\mathrm{proj}}, \boldsymbol{x})}{\operatorname{norm}(Q^{\mathrm{proj}})}
\le |Q(x)| \cdot e^{\frac{\log(n+1)}2 \delta}.
\]
\end{enumerate}
\end{proof}

\begin{Lem}\label{Lemma Dist implies l_2 distance}
Let $x=(x_1,\ldots,x_n) \in \mathbb{C}^n$, and write $\boldsymbol{x} := (1 : x) \in \mathbb{P}^n(\mathbb{C})$. Let $d>1$ be some real number.

Then, for every $r \in \left(0, \frac{1-\frac1d}{\sqrt{1+\sum_{i=1}^n|x_i|^2}\cdot \sqrt{n+1}}\right]$, we have
\[
    \overline{B_{\boldsymbol{x}, \operatorname{Dist}}(r)}
    \subseteq
    \{\boldsymbol{y}=(1:y)\in \mathbb{P}^n(\mathbb{C}) \mid y\in \mathbb{C}^n, \|x-y\|_2\le c_dr\}.
\]
where
\begin{equation}\label{Equation c}
c_d:=\sqrt{1+\sum_{i=1}^n|x_i|^2}\cdot \sqrt{1+d^2\cdot n\cdot\max_i\{|x_i|^2\}}.
\end{equation}
\end{Lem}
\begin{proof}

We denote $d_1:=\sqrt{1+\sum_{i=1}^n|x_i|^2}\cdot \sqrt{n+1}$.

Let $\boldsymbol{y} \in \overline{B_{\boldsymbol{x}, \operatorname{Dist}}(r)}$. Write
\[
    \boldsymbol{y} =: (b_0 : b_1 : \cdots : b_n).
\]
Without loss of generality, we may assume that $\max_i |b_i| = 1$ (by rescaling). Let $i_{\max}$ be such that $|b_{i_{\max}}| = 1$. 
The definition of projective distance together with the choice of $\boldsymbol{y}$ implies for each $i$
\[
    \frac{|b_i - b_0x_i |}{d_1}\le \operatorname{Dist}(\boldsymbol{x}, \boldsymbol{y}) \le r.
\]
In particular, for $i = i_{\max}$, this implies
\begin{equation} \label{Equation bound of 1-xb}
    |1 - b_0  x_{i_{\max}}| \le d_1r,
\end{equation}
and $b_0\neq 0$ since $d_1r<1$.
Therefore, we may normalize $\boldsymbol{y}$ as
\[
    \boldsymbol{y} = (1 : y_1 : \cdots : y_n),
\]
with all $y_i := \frac{b_i}{b_0}$. 

The choice of $r$ with \iref{Equation bound of 1-xb} implies
\begin{align*}
|b_0|
    \geq
\dfrac{1-d_1r}{|x_{i_{\max}}|}
    \geq
\dfrac1{d|x_{i_{\max}}|},
\end{align*}
and
\begin{equation} \label{Equation upper bound yi} 
\max_{1\le i \le n}\{|y_i|\}\leq d|x_{i_{\max}}|.
\end{equation}
Using again the projective distance estimate together with \iref{Equation upper bound yi}, we have
\[
r
    \geq
\dfrac{\sqrt{\sum_{1\le i \le n}|x_i - y_i|^2}}{\sqrt{1+\sum_i|x_i|^2}\sqrt{1+\sum_{i}|y_i|^2}}
    \geq
 \dfrac{\sqrt{\sum_{1\le i \le n}|x_i - y_i|^2}}{c_d}.
\]
Therefore,
\[
    \| (x_1, \ldots, x_n) - (y_1, \ldots, y_n) \|_2\le c_dr,
\]
and $\boldsymbol{y}$ lies in the projectivization of the $l_2$-ball of radius $c_dr$ centered at $x$.
\end{proof}

\begin{Lem}\label{Lemma 2}
Let $\omega = (\omega_1, \ldots, \omega_n) \in \mathbb{C}^n$, let $d>1$ be some real number, and let $r>0$ be a real number satisfying:

\[
    r\leq\min\left(\frac{1-\frac1d}{\sqrt{1+\sum_{i=1}^n|\omega_i|^2}\cdot \sqrt{n+1}},\frac1{\sqrt{1+\sum_{i=1}^n|\omega_i|^2}\cdot \sqrt{1+d^2\cdot n\cdot\max_i\{|\omega_i|^2\}}}\right).
\]

Let $\delta, \tau\geq1$ be real numbers, and let
\[
Q \in \mathbb{Z}[x_1, \ldots, x_n]
\]
be a non-zero polynomial of degree at most $\delta$ and log-height at most $\tau$.
Set
\begin{align*}
\boldsymbol{\omega} &:= (1 : \omega_1 : \cdots : \omega_n) \in \mathbb{P}^n(\mathbb{C}), \\
Q^{\mathrm{proj}} &:= X_0^{\deg Q} \cdot Q\left(\frac{X_1}{X_0}, \ldots, \frac{X_n}{X_0} \right).
\end{align*}
Define
\[
c_d' = 2n+1 + \log\left(1+d^2\cdot n\cdot\max_i\{|\omega_i|^2\}\right).
\]
Assume that
\[
e^{c_d'(\delta + \tau)} r<|Q(\omega)|.
\]
Then $Q^{\mathrm{proj}}$ does not vanish on the closed projective ball $\overline{B_{\boldsymbol{\omega}, \operatorname{Dist}}(r)}$.
\end{Lem}

\begin{proof}
Let $c_d$ be as in \iref{Equation c}:

\[
c_d:=\sqrt{1+\sum_{i=1}^n|\omega_i|^2}\cdot \sqrt{1+d^2\cdot n\cdot{\max_i\{|\omega_i|^2\}}}.
\]

By \iref{Lemma Dist implies l_2 distance}, it suffices to show that $Q^{\mathrm{proj}}$ has no zero in the set
\[
    \left\{ \boldsymbol{y} = (1 : y_1 : \cdots : y_n) \in \mathbb{P}^n(\mathbb{C}) \;\middle| \;y \in \mathbb{C}^n, \;\|\omega-y\|_2 \le c_dr \right\}.
\]
Let $y \in \mathbb{C}^n$ satisfy $\|\omega-y\|_2 \le c_dr$ (so that $\|\omega-y\|_\infty\leq c_dr$ as well). Consider the Taylor expansion of $Q$ at the point $\omega$. We have
\[
    Q(y)-Q(\omega) = \sum_{\substack{\alpha = (\alpha_1, \ldots, \alpha_n) \\ \exists i, \;\alpha_i \ge 1}} q_\alpha \prod_{i=1}^n (y_i - \omega_i)^{\alpha_i},
\]
where
\[
    q_\alpha = \frac1{\alpha!} \left( \frac{\partial}{\partial x} \right)^\alpha Q(\omega).
\]

We now estimate $|q_\alpha|$. Write $Q =: \sum_{\beta} Q_{\beta} x^{\beta}$. Then, by definition,
\[
    \frac1{\alpha!} \left( \frac{\partial}{\partial x} \right)^\alpha Q 
    = \sum_{\substack{\beta=(\beta_1,\ldots,\beta_n) \\ \beta \ge \alpha}} \prod_{i=1}^n \binom{\beta_i}{\alpha_i} Q_{\beta} x^{\beta - \alpha}.
\]
We will now use the estimates
\begin{align*}
    \binom{\beta_i}{\alpha_i} &\le 2^{\beta_i}, \\
    |Q_{\beta}| &\le e^{\tau},
\end{align*}
to deduce the following bound:
\begin{align*}
    |q_\alpha| 
    \le (\delta + 1)^n 2^\delta e^\tau \left(\max_{1 \le i \le n}\{\max(1,|\omega_i|)\} \right)^{\delta}.
\end{align*}
Since $\|x-y\|_\infty \le c_dr\leq1$, we obtain
\begin{align*}
    |Q(y)-Q(\omega)|
        &\leq 
    (\delta + 1)^n |q_\alpha| \cdot c_dr \\
        &\leq
    (\delta+1)^{2n}\cdot2^\delta\cdot e^\tau\cdot\left(\max_{1 \le i \le n}\{\max(1,|\omega_i|)\}\right)^{\delta}\cdot c_dr\\
        &\leq
    e^{c_d'(\delta + \tau)} r,
\end{align*}
which, together with the assumption that $e^{c_d'(\delta + \tau)} r < |Q(\omega)|$, implies that $Q(y)\neq 0$.

Hence, $Q^{\mathrm{proj}}$ has no zero in $\overline{B_{\boldsymbol{x}, \operatorname{Dist}}(r)}$.
\end{proof}

\bibliographystyle{plain}
\bibliography{references}

\end{document}